\documentclass[a4paper,leqno,11pt]{amsart}

\usepackage[english]{babel}

\usepackage{hyperref}

\usepackage{amsfonts}
\usepackage{amsmath}
\usepackage{amsthm}
\usepackage{amssymb}
\usepackage{indentfirst}
\usepackage{color}
\usepackage{tabularx}
\usepackage{graphicx}
\usepackage{esint}
\allowdisplaybreaks

\usepackage[a4paper,top=3.5cm,bottom=1cm,left=2.5cm,right=2.5cm]{geometry}

\usepackage{subcaption}

\usepackage{mathrsfs}
\usepackage{mathtools}
\usepackage{enumitem}
\usepackage[normalem]{ulem}

\usepackage[utf8]{inputenc}
\usepackage[T1]{fontenc}
\usepackage[english]{babel}

\numberwithin{equation}{section}

\theoremstyle{plain}
\begingroup
\theoremstyle{plain}
\newtheorem{theorem}{Theorem}[section]
\newtheorem{corollary}[theorem]{Corollary}
\newtheorem{proposition}[theorem]{Proposition}
\newtheorem{lemma}[theorem]{Lemma}
\theoremstyle{definition}

\theoremstyle{remark}
\newtheorem{remark}[theorem]{Remark}
\endgroup

\theoremstyle{definition}
\theoremstyle{remark}

\mathchardef\emptyset="001F

\newcommand{\stress}{\boldsymbol{\sigma}} 
\newcommand{\strain}{\boldsymbol{\varepsilon}} 

\newcommand{\R}{\mathbb{R}}
\newcommand{\E}{\mathcal{E}}

\newcommand{\C}{\mathbb{C}}
\newcommand{\e}{\mathrm{E}}
\newcommand{\id}{\mathit{id}}
\newcommand{\G}{\mathcal{G}}

\newcommand{\di}{\mathrm{d}}

\newcommand{\A}{\mathcal{A}}
\newcommand{\HH}{\mathcal{H}}
\newcommand{\Om}{\Omega}

\newcommand{\M}{\mathbb{M}}

\newcommand{\F}{\mathcal{F}}

\newcommand{\J}{\mathcal{J}}
\newcommand{\p}{\boldsymbol{p}}

\newcommand{\argmin}{\mathrm{argmin}}

\definecolor{dred}{rgb}{.8,0,0}
\definecolor{ddmagenta}{rgb}{0.7,0,0.9}
\definecolor{ddcyan}{rgb}{0,0.2,1.0}
\definecolor{Orchid}{rgb}{0.7,0.4,0}

\newcommand{\EEE}{\color{black}}

\newcommand{\coloneq }{\hspace{1pt}\raisebox{0.74pt}{\scalebox{0.8}{:}}\hspace{-2.2pt}=}

\newcommand{\ha}{\mathbb{H}}
\newcommand{\q}{\boldsymbol{\mathrm{Q}}}
\newcommand{\qq}{\boldsymbol{q}}
\newcommand{\eeta}{\boldsymbol{\eta}}

\title[Topology optimization for incremental
elastoplasticity]{Topology optimization for incremental
  elastoplasticity:  a phase-field approach}
\author[S. Almi]{Stefano Almi}
\address[Stefano Almi]{Faculty of Mathematics, University of Vienna, 
Oskar-Morgenstern-Platz 1, 1090 Wien, Austria.}
\email{stefano.almi@univie.ac.at}
\urladdr{http://www.mat.univie.ac.at/$\sim$almi}

\author[U. Stefanelli]{Ulisse Stefanelli} 
\address[Ulisse Stefanelli]{Faculty of Mathematics, University of
  Vienna, Oskar-Morgenstern-Platz 1, A-1090 Vienna, Austria,
Vienna Research Platform on Accelerating
  Photoreaction Discovery, University of Vienna, W\"ahringerstra\ss e 17, 1090 Wien, Austria,
 \& Istituto di
  Matematica Applicata e Tecnologie Informatiche {\it E. Magenes}, via
  Ferrata 1, I-27100 Pavia, Italy
}
\email{ulisse.stefanelli@univie.ac.at}
\urladdr{http://www.mat.univie.ac.at/$\sim$stefanelli}

\date{\today}

\begin{document}
\maketitle

\begin{abstract}
 We discuss a topology optimization problem for an
elastoplastic medium. The distribution of material in a region is
optimized with respect to a given target functional taking into
account compliance. The incremental elastoplastic problem serves as
state constraint. We prove that the  topology
optimization problem admits a solution. First-order optimality
conditions are obtained by considering a regularized problem and
passing to the limit.
\end{abstract}

\section{Introduction}
\label{s:intro}

Topology Optimization is concerned with determining the optimal  
shape of a mechanical piece against a number of criteria. Besides mechanical
performance, these criteria may include weight, manufacturing
costs, topological, and geometrical features. The distribution of
material within an a-priori given region $\Omega\subset {\mathbb R}^n$ is the control parameter of
the process. Given the portion $E\subset \Omega$ to be filled with
material, one determines the mechanical response of the body and
minimizes a target functional depending on $E$ and such response. 
This very general setting
applies to a number of different shape design problems, from
mechanical engineering, to aerospace and 
automotive, to architectural engineering, to biomechanics \cite{MR2008524}.

The applicative interest in topology optimization has triggered an
intense research activity, which in turn generated a wealth of
results at the engineering and computational level. Starting from the
pioneering paper \cite{Bendsoe0}, see also \cite{Allaire}, the literature has developed to
cover a wide range of different mechanical settings \cite{Zhang},
including strain-gradient theories \cite{Li}, finite strains \cite{Wallin}, thermoelasticity
\cite{Zhang2}, material interfaces \cite{Vermaak}, surface effects
\cite{Nanthakumar}, graded materials \cite{Carraturo}, stochastic
effects \cite{Carrasco}, and fluid-structure interactions
\cite{Lundgaard}. The reader is referred to the monographs
\cite{Deaton,Zhu} for additional material. Correspondingly, extensive
numerical experimental campaigns have been developed. Among the
different 
computational approaches in use one can mention finite elements~\cite{Bruggi,Petersson}, NURBS~\cite{Gao},  smoothed-particle hydrodynamics~\cite{Lin},
shape-derivatives~\cite{Sokolowski}, and level-set
methods~\cite{Allaire2}.

On the more theoretical side, 
the linear elastic setting has been considered in a number of
contributions. In \cite{Bourdin} the existence of optimal
shapes is tackled by introducing a penalization of interfaces between
solid $E$ and void $\Omega \setminus E$ and considering an additional {\it phase-field} regularization. The actual position
of the body is hence modeled as a level set of a scalar field $z\in
[0,1]$ and gradients of this parameter are penalized. By removing such
penalization, a solution of the original {\it sharp-interface} limit
is recovered by means of a $\Gamma$-convergence argument~\cite{DalMaso}. Existence under additional stress constraints is
discussed in \cite{Burger} and the extension to hyperelasticity is presented
in \cite{Penzler}.
More recently, 
first-order optimality conditions have been obtained in the linear
elastic setting
\cite{Blank2,Blank1,Carraturo} for both the sharp-interface
model and its phase-field approximation, also for graded materials.

Inelastic effects such as large or cyclical stresses,  
permanent deformations, damage, and
fracture appear ubiquitously in applications. Topology optimization in
the inelastic setting is for instance paramount to the bending, punching, and
machining of steel sheets, which are tasks of the utmost applicative relevance.
Correspondingly, topology optimization in inelastic settings 
has already attracted strong attention in the
engineering community, see the pioneering \cite{Swan} and
\cite{Amir,James,Liu,Maute,Tortorelli,Wallin} among others. 

 From a mathematical standpoint, first results of existence of optimal shapes have been provided in~\cite{haslinger1, haslinger2, Hlavacek1, Hlavacek3} for elasto-perfectly plastic bodies in a static regime. The hardening effect has been considered in~\cite{Pistora}, while~\cite{Hlavacek2} dealt with shape optimization in an evolution setting. All the mentioned papers deal with two-dimensional problems. Furthermore, they a priori assume a Lipschitz regularity of the unknown optimal shape, in order to guarantee compactness and therefore existence of solutions. The beam structure and frame optimization was considered in~\cite{Karkauskas, Khanzadi, Pedersen}, where again only an existence result was shown. Some first optimality conditions in terms of shape derivatives appeared in~\cite[Chapters~4.8 and 4.9]{Sokolowski2} for an elastic torsion problem and for the viscoplastic model of Perzyna. The latter framework has been also considered in the recent paper~\cite{Maury} as a regularization of static perfect plasticity. For such a regularized model, in~\cite{Maury} the authors consider a volume and surface optimization problem. In this setting, the shape gradients are computed, under the usual smoothness assumptions on the shape~$\Om$ and on its variation. The shape derivatives are then combined with a level set method in order to numerically solve the shape optimization problem in presence of plasticity. However, the theoretical existence of solutions and the passage from the regularized to the original perfect plasticity problem are not fully discussed.

 The focus of this paper is on the phase-field approach to
 topology optimization  in the setting of incremental, linearized
elastoplasticity with hardening. Referring to Section \ref{s:setting}
for all necessary assumptions and details,  our starting point is the \EEE {\it
  sharp-interface} problem
\begin{equation}\label{min}\min_{z}\{\J(z,u) \ : \ (u,\p) \in \argmin \,\E(\cdot,z)\}
\end{equation}
where the compliance-type {\it target functional}  is defined as
$$
\J(z, u) \coloneq \int_{\Omega} \ell(z) f{\, \cdot\,} u \, \di x +
\int_{\Gamma_{N}}  g { \, \cdot\, } u \, \di \HH^{n-1} + \frac{1}{6} \, \mathrm{Per}( \{z=1\}; \Om) \,,$$
and the {\it incremental elastoplastic state functional} reads
\begin{align*}
\E( z, u, \p) \coloneq & \  \frac{1}{2} \int_{\Om} \C(z)  ({\rm E}u-\p)  {\, \cdot} ({\rm E}u-\p)
\, \di x + \frac{1}{2} \int_{\Om} \ha(z) \p{\,\cdot \,} \p \, \di x + \int_{\Om} d(z) | \p | \, \di x
\\
&
 - \int_{\Om}  \ell(z)  f {\, \cdot \, } u \, \di x - \int_{\Gamma_{N}} g {\, \cdot\, } u \, \di \HH^{n-1} \,. \nonumber
\end{align*}
The  function  $z \colon \Omega \to \{0,1\}$   is a phase
indicator, distinguishing between two different elastoplastic
media. The phase $\{z=1\}$ corresponds to the body
to be determined via the optimization process, whereas the
phase $\{z=0\}$ is assumed to be a very soft
material filling the complement of the container $\Omega$, see
Subsection \ref{sec:const}. \EEE
 The smooth nonnegative function $ \ell(\cdot)$ represents the density of
the body, in dependence of $z$. In particular, $\ell(0)=0$.
We assume  the whole container $\Omega$ to deform under the action
of the body force $ \ell(z)  f$ and the boundary traction $g$. The
elasticity tensor~$\C$, the linear hardening tensor~$\ha$, and the
yield stress $d$ hence depend on~$z$,  see Section
\ref{sec:assumptions}.  The fields
$u\colon \Omega \to {\mathbb R}^n$ and $ \p \colon \Omega \to {\mathbb R}^{n\times
  n}$  are the {\it displacement} and the {\it plastic strain},
respectively, and ${\rm E}u=(\nabla u + \nabla u^T)/2$ is the
symmetric {\em total strain}.

Given   $z$, \EEE the minimization of the incremental elastoplastic state
functional $\E$ gives a unique elastoplastic state $(u,\p)$. Note that
$\E$ is nothing but the {\it complementary elastic energy} of the body,
augmented by the linear {\it dissipation} term $d(z)|\p|$, modeling the
plastic work from a prior unplasticized state. Then,
the displacement component $u$ enters into the definition of $\J$, which
ultimately depends just on $z$. Note that the functional $\J$ features
the total-variation norm of
$z$, which indeed corresponds to the perimeter of   the phase \EEE $\{z=1\}$
in $\Omega$. In particular, the boundedness of $\J$ implies that $z\in
BV(\Omega;\{0,1\})$, the space of functions of {\it bounded variation}~\cite{MR1857292}.

 Besides the sharp-interface problem \eqref{min}, we also consider its
{\it phase-field} approximation 
\begin{equation}\label{min2}\min_{z}\{\J_\delta(z,u) \ : \ (u,\p) \in \argmin \,\E(\cdot,z)\}
\end{equation}
where, for $\delta >0$, the phase-field target functional $\J_\delta$
is defined as
$$ 
\J_{\delta} (z, u) \coloneq \int_{\Om} \ell(z) f {\, \cdot \,} u \,
\di x + \int_{\Gamma_{N}}  g { \, \cdot \, } u \, \di \HH^{n-1} +
\int_{\Om} \frac{\delta}{2} | \nabla{z} |^{2}  \, \di x+
\int_{\Om}  \frac{ z^{2}
  (1-z)^{2}}{ 2 \delta} \, \di x\,,
$$
 where the {\it phase-field} function $z$ is  now assumed to belong to
$H^1(\Omega;[0,1])$. 

We prove that problems \eqref{min} and \eqref{min2} admit solutions,
see Propositions \ref{p:1.1} and \ref{p:1}, respectively.  In
addition, we prove that optimal $z_\delta$ solving problem \eqref{min2}
converge, up to subsequences, to a solution to the sharp-interface problem
\eqref{min}, see Corollary \ref{cor}. Indeed, this relies on 
a general $\Gamma$-convergence result as $\delta \to 0$, see
Theorem \ref{t:gamma_conv}.

The phase-field problem~\eqref{min2} is then further regularized by
replacing the dissipation term $d(z)|\p|$ by $d(z)h_\gamma(\p)$ in~$\E$,
where~$h_\gamma$ is a smooth approximation of the norm, namely
$h_\gamma(\p) \to |\p|$ as $\gamma \to +\infty$. The ensuing 
regularized phase-field problem admits solutions
$z_{\delta}^{\gamma}$. For fixed $\delta>0$, as $\gamma \to +\infty$ the sequence $z_{\delta}^{
  \gamma}$ converges, up to subsequences, to a solution of the
phase-field problem~\eqref{min2}, see Proposition~\ref{p:2}. 

In addition,
owing to the smoothness of the corresponding control-to-state map for
all $\gamma$, one
can derive first-order optimality conditions for the regularized
phase-field problem, see Theorem~\ref{t:regoptim} and Corollary~\ref{c:regoptim}. Eventually, by letting
$\gamma\to +\infty$ we derive first-order
optimality conditions for the phase-field problem~\eqref{min2}, see
Theorem~\ref{t:lim_optim}. 

 Compared with available results, the novelty of our contribution
is twofold. On the one hand, we provide rigorous existence and
convergence results for phase-field approximations including first-order
optimality conditions. On the other hand, we focus here on linear
kinematic hardening and optimize against compliance. 

Before closing this introduction, let us mention that the literature on
existence and optimality conditions for elastoplastic problems is
rather scant and, to our knowledge, always focusing on force and
traction  control. The incremental elastoplastic problem is discussed in
\cite{delosReyes,Herzog2,Herzog3} whereas the 
quasistatic setting is addressed in the series
\cite{Wachsmut1,Wachsmut2, Wachsmut3}, see also \cite{sww}. Some
related results on elastoplasticity are in \cite{Brokate}, whereas
\cite{Adam} and \cite{Eleuteri,Eleuteri2} deal with quasistatic
adhesive contact and shape memory alloys, respectively. The reader is
referred to \cite{Rindler,Rindler2} for some abstract optimal control
theory for rate-independent systems.

As concerns our quest for first-order optimality conditions, our
analysis follows some argument from \cite{delosReyes}. Here, the
optimal control problem in incremental elastoplasticity is discussed, with  elastoplastic state $(u,\p)$ controlled via $f$ and
traction $g$ and the coefficients $\C$, $\ha$, and $d$ are given
constants. In particular, the  incremental elastoplastic state
functional $\E$ in \cite{delosReyes} is linear with respect to the
controls. Here the control $z$ appears nonlinearly in $\E$
instead.   This generates some
additional difficulty which we overcome in a series of technical
lemmas in Section~\ref{s:reg} below.
 
This is the plan of the paper: we  review the mechanical model,
specify  our setting, and state and
prove existence of both sharp-interface problems and phase-field
approximations in Section~\ref{s:setting}, where we also discuss
$\Gamma$-convergence  as $\delta \to 0$.  The $\gamma$-regularized problem is then
discussed in Section~\ref{s:reg}, where the corresponding first-order
optimality conditions are derived. Eventually, in Section~\ref{s:optimality} we
pass to the limit as $\gamma \to +\infty$ and obtain first-order
optimality conditions for the phase-field approximation.


\section{Setting of the problem}
\label{s:setting}
 
 Our focus is on linearized elastoplastic evolution with kinematic
hardening. In particular, we concentrate on the incremental setting, also
called static, where the material is assumed to evolve from a
pristine, unplasticized  state \cite{Han}.

\subsection{Constitutive model}\label{sec:const}
Let~$\Om$ be an open bounded subset of~$\R^{n}$ with Lipschitz
boundary~$\partial\Om$ containing the elastoplastic 
specimen. Although $n=3$ suffices for the application, our
analysis is independent of $n\in \mathbb N$, which we hence keep general. 
The   body to be identified via the topology optimization is described \EEE by  the level sets of the scalar field~$z$   at each point.  In the {\it sharp-interface} case, we
  have $z\in \{0,1\}$ and the phase $\{z=1\}$ represents the
body to be identified. In the {\it phase-field} case, such 
body corresponds to the region $\{z>0\}$. In both cases, the phase
$\{z=0\}$ represents the complement in the design domain $\Om$ of the body
to be identified and  is here interpreted as very compliant
elastoplastic medium. \EEE
  This approach is indeed
classical \cite{Allaire2002} and allows for a sound mathematical
treatment. \EEE  In particular, all state quantities are assumed to be
defined in the whole container~$\Omega$.

Let $\M^n, \, \M^n_S$, and $ \M^n_D$ denote the space of $n\times n$ 
matrices, symmetric matrices, and symmetric and deviatoric (trace
free) matrices, respectively.   We use the symbol $\cdot$ to
indicate both the classical contraction product between matrices
$\boldsymbol A\cdot \boldsymbol B = \sum_{i,j}A_{ij}B_{ij}$, as well
as the scalar product between vectors $u\cdot v = \sum_i u_iv_i$. \EEE

The elastoplastic state of the material is determined by prescribing
its {\it displacement} $u \colon \Om \to \R^n$ from some reference configuration and its {\it
  plastic strain} $\p \colon \Om \to\M^n_D$. 
By indicating by ${\rm E}u$ the {\it total strain} of the material, namely ${\rm E}u = (\nabla u + \nabla
u^T)/2$, we assume the classical \cite{Simo-Hughes} {\it additive strain
  decomposition}
\begin{equation}
{\rm E}u = \strain + \p\label{eq:decompose}
\end{equation}
where $\strain \in \M^n_S$ is the {\it elastic strain}, directly related to the
stress state of the material, whereas~$\p\in \M^n_D$ records the plasticization
of the material.

The {\it internal energy} density of the medium is classically
\cite{Lamaitre} given by
$$ \Psi(\strain,\p)=\frac12 \mathbb C(z) \strain \cdot  \strain +
\frac12 \mathbb H(z) \p \cdot \p.$$
The first term above is the {\it elastic energy} density whereas the
second corresponds to the {\it kinematic hardening} potential. Correspondingly, $\mathbb C(z)$ and  $\mathbb H(z)$ are the {\it elastic} and the
{\it hardening} tensor, respectively, two symmetric $4$-tensors,
depending on the scalar field~$z$.

By computing the variations of the internal energy density $\Psi$ with
respect to its arguments we obtain the {\it constitutive relations}
\begin{equation}
\stress =\mathbb C(z) \strain, \quad \boldsymbol  \chi=
-\stress+ \mathbb H(z) \p. \label{eq:const}
\end{equation}
Here, $\stress \in \M^n_S$ and $\boldsymbol \chi \in \M^n_S$  are the {\it
  stress} and the {\it thermodynamic force} associated with $\p$. In
particular, the first relation models the {\it linear elastic
  response} of the medium whereas the second features the {\it
  back-stress} $\mathbb H(z) \p$ due to the kinematic hardening
\cite{Lamaitre}.

The elastoplastic process results from the interplay of energy-storage
and dissipation mechanisms. In the incremental case, the constitutive relations
\eqref{eq:const} are indeed complemented by the {\it normality rule}
\cite{Han} based on Von Mises criterion
\begin{equation} \boldsymbol \chi \in d(z)\partial |\p|.\label{eq:normality}
\end{equation}
Here, $d(z)>0$ is the {\it yield stress} activating plasticization,
which is allowed to depend on  $z$ as well. The symbol $\partial |\p|$
denotes the subdifferential of the convex function $|\cdot|$ at
$\p$. In particular, $\partial |\p| = \p/|\p|$ if $\p \not =
\boldsymbol 0$ and $\partial |\boldsymbol 0| = \{\boldsymbol r \in \M^n_D \ : \
|\boldsymbol r |\leq 1\}$.

The normality rule \eqref{eq:normality} can be equivalently expressed
in the classical {\it complementarity form} by defining the {\it yield function}
$F(\boldsymbol \chi ,z)= |\boldsymbol \chi | - d(z)$ and letting
$$\p= \zeta \frac{\partial F(\boldsymbol \chi ,z)}{\partial
  \boldsymbol \chi}, \quad  F(\boldsymbol \chi ,z)\leq 0, \quad \zeta
\geq 0, \quad \zeta \, F(\boldsymbol \chi ,z)=0\,.$$
Note that the yield function $F$ is differentiable except in $\boldsymbol
\chi=\boldsymbol 0$, where nonetheless $\zeta=0$ since $F(\boldsymbol
0 ,z) = -d(z)<0$.

\subsection{Assumptions on the material parameters}\label{sec:assumptions}

Before moving on, let us collect here the assumptions on the material
parameters that will be used throughout, without further explicit
mention.

We start by assuming the elasticity tensor $\C$ and the
kinematic-hardening tensor $\ha$ to be {\it isotropic} for all $z$. In
particular, we ask for  
$$\C( z )  := 2\mu(z) \mathbb I + \lambda(z) ( \boldsymbol{\mathrm{I}}
\otimes  \boldsymbol{\mathrm{I}}), \quad  \ha(z):= h(z) \mathbb I  $$
where $\lambda$ and $\mu$ are the Lam\'e coefficients, $h$ is the
hardening modulus, and $\mathbb I$ and $\boldsymbol{\mathrm{I}}$ denote the identity 4 and
2-tensor, respectively. Isotropy in particular guarantees that $\C$ and $\ha$
map $\M^n_D$ to $\M^n_D$.

In what follows, we assume the material coefficients to be
differentiable with respect to $z$ and to be defined in all of
$\mathbb R$. In particular, we
ask 
$$
\mu, \,  \lambda,  \, h, \, d \in C^1(\mathbb R)
$$
 and impose them to be
constant on $\{z\leq 0\}$ and $\{ z\geq 1\}$.  This last requirement
 is of technical nature and 
allows us to drop the usual constraint $z \in [0,1]$ as it can be obtained a posteriori without changing the
solution of the problem
(see~\eqref{e:optimization1}--\eqref{e:optimization2} for further
details). 

Moreover, we assume all coefficients to be positive and bounded, uniformly with
respect to $z$, namely, 
$$\exists 0<\alpha<\beta<+\infty \, \forall z \in [0,1]: \quad \alpha \leq \mu(z),
\, \lambda(z), \, h(z), \,  d(z) \leq \beta\, .$$
This in particular implies that $\C$ and $\ha$ are uniformly
positive definite and bounded, independently of $z$, namely,
\begin{eqnarray}
&& \displaystyle \alpha_{\C} |\boldsymbol{\e} |^{2} \leq \C( z )
   \boldsymbol{\e} {\, \cdot \,} \boldsymbol{\e} \leq \beta_{\C} |
   \boldsymbol{\e} |^{2} \qquad \forall z  \in [0,1] ,\, \forall \boldsymbol{\e} \in \M^{n}_{S} \,, \label{e:1} \\ [1mm]
&& \displaystyle \alpha_{\ha} | \boldsymbol{\q} |^{2} \leq \ha(  z  )
   \boldsymbol{\q} {\, \cdot \,} \boldsymbol{\q} \leq \beta_{\ha} |
   \boldsymbol{\q} |^{2} \qquad \forall z  \in [0,1], \, \forall \boldsymbol{\q} \in \M^{n}_{D}\,, \label{e:2}
\end{eqnarray}
for some $0< \alpha_{\C} < \beta_{\C}<+\infty$ and $0 <
\alpha_{\ha} < \beta_{\ha} <+\infty$. 



In addition to the above, we let $\ell \in
C^1(\mathbb R)$ monotone nondecreasing with $\ell(z) = 0$ for $z \leq 0$ and
$\ell(z) =1$ for $z\geq 1$ express the {\it density} of
the   body to be identified, \EEE at a given value of the scalar field~$z$. In particular, observe that $\ell(z) = z$ in the sharp-interface case
$z\in \{0,1\}$.

Examples of material laws complying with the above assumption are 
\begin{align*}
  \mu(z) = \mu_0 + \mu_1 \ell(z),  \ \ 
\lambda(z) = \lambda_0 + \lambda_1 \ell(z),  \ \ 
h(z) = h_0 + h_1 \ell(z),  \ \ d(z)=d_0+d_1\ell(z)
\end{align*}
where the positive parameters $\mu_1$, $\lambda_1$, $h_1$, $d_1$ and
$\mu_0$, $\lambda_0$, $h_0$, $d_0$ correspond to the elastoplastic
  body to be identified \EEE and to the soft medium surrounding
  it, \EEE respectively. In particular, the
standard situation would be $0<\mu_0 \ll \mu_1$, $0<\lambda_0 \ll \lambda_1$, and
$0<h_0 \ll h_1$. On the contrary,~$d_0$ and~$d_1$ need not be ordered: if~$\mu_0$,~$\lambda_0$, and~$h_0$ are small, large values of the strain can be achieved under moderate stresses, so that the soft medium does not easily plasticize. The same effect can be achieved by choosing the yield stress $d_0 \gg 1$.

\subsection{ Equilibrium problem}  Let us now complement the
constitutive model \eqref{eq:const}--\eqref{eq:normality} with the
equilibrium system and boundary conditions. In particular, we assume
the displacement to be prescribed on the Dirichlet part $\Gamma_D$ of
the boundary $\partial \Om$. Thus, we will impose $u=w$ on
$\Gamma_D$ for some given $w\in H^{1}(\Om;\R^{n})$. Moreover, we
assume  a given traction $g \in L^2(\Gamma_N;\R^n)$ to be exerted on the Neumann part
$\Gamma_N$ of the boundary $\partial \Om$.  More precisely, let $\Gamma_D,\, \Gamma_N \subset \partial \Omega $ be open in
the topology of $\partial \Omega$ with $\overline \Gamma_N \cap \overline \Gamma_D=\emptyset$, where $\overline \Gamma_N$ and $
\overline \Gamma_D$ are closures in~$\partial \Omega$. We moreover
assume that $\Gamma_D$ has positive surface measure, namely  $\mathcal
H^{n-1} (\Gamma_D)>0$, where the
latter is the $(n{-}1)$-Hausdorff surface measure in $\R^n$. 
Furthermore, we assume that $\Omega \cup \Gamma_{N}$ is regular in the sense of Gr\"oger~\cite[Definition~2]{MR990595}, that is, for every $x \in \partial\Om$ there exists an open neighborhood~$U_{x} \subseteq \R^{n}$ of~$x$ and a bi-Lipschitz map~$\Psi_{x} \colon U_{x} \to \Psi(U_{x})$ such that $\Psi_{x} ( U_{x} \cap (\Om \cup \Gamma_{N}))$ coincides with one of the following sets:
\begin{eqnarray*}
&& \displaystyle V_{1} \coloneq \{ y \in \R^{n} : \, | y | \leq 1 , \, y_{n} <0\}\,, \\ [1mm]
&& \displaystyle V_{2} \coloneq \{ y \in \R^{n} : \, | y | \leq 1, \, y_{n}\leq 0\} \,, \\ [1mm]
&& \displaystyle V_{3} \coloneq \{ y \in V_{2} ; \, y_{n} <0 \text{ or } y_{1} >0\} \,,
\end{eqnarray*}
where $y_{i}$ is the $i$-th component of~$y \in \R^{n}$. This last
assumption   is crucially used   in the proof of
Theorem~\ref{t:regoptim}.

All in all, given $z \colon \Omega \to [0,1]$, the equilibrium of the system
corresponds to solving the differential problem
\begin{align}
 & \nabla \cdot \stress  + \ell(z) f =0 \quad \text{in} \  \Omega,\label{eq:10}\\
&\stress \nu =g \quad \text{on} \  \Gamma_N,\label{eq:12}\\
&u = w \quad \text{on} \  \Gamma_D,\label{eq:13}\\
&\stress = \C(z)\strain \quad \text{in} \  \Omega, \label{eq:14}\\
&  d(z) \partial |\p| + \ha(z)\p \ni \stress \quad  \text{in} \  \Omega,\label{eq:100}
\end{align}
where $\nu$ denotes the outward unit normal to $\Gamma_N$. The 
term $\ell(z)f$ corresponds to the applied body force, for $\ell(z)$
is the density and $f$ is the applied body force per unit density.

\subsection{ Variational formulation}  We now introduce a
variational formulation for \eqref{eq:10}--\eqref{eq:100}. 
Let us define the set of admissible states~$\A(w)$ as
\begin{equation}\label{e:A}
\A(w) \coloneq \{ (u, \strain, \p) \in H^{1}(\Om;\R^{n}) \times L^{2}(\Om; \M^{n}_{S}) \times L^{2}(\Om; \M^{n}_{D}): \, \e u = \strain + \p, \, u=w \text{ on~$\Gamma_{D}$} \}\,.
\end{equation}
 Note that the additive strain decomposition 
\eqref{eq:decompose} as well as the Dirichlet boundary condition~\eqref{eq:13} are included in the definition of $\A(w)$.


Given a state~$(u, \strain, \p) \in \A(w)$ and  a field 
 $z \in L^{\infty}(\Om)$, we define the  {\it
   incremental elastoplastic state 
   functional} 
\begin{align}\label{e:energy}
\E( z, u, \strain, \p) \coloneq & \  \frac{1}{2} \int_{\Om} \C(z)   \strain {\, \cdot \,} \strain \, \di x + \frac{1}{2} \int_{\Om} \ha(z) \p{\,\cdot \,} \p \, \di x + \int_{\Om} d(z) | \p | \, \di x
\\
&
 - \int_{\Om}  \ell(z)  f {\, \cdot \, } u \, \di x - \int_{\Gamma_{N}} g {\, \cdot\, } u \, \di \HH^{n-1} \,. \nonumber
\end{align}

 Owing to our assumptions on $\C$, $\ha$, $d$, and $\ell$,
we have that $\E(z, u, \strain, \p) = \E( 0 \vee z \wedge 1, u,
\strain, \p)$ for every $z \in L^{\infty}(\Om)$. Thus, the constraint~$z \in [0,1]$ needs not to be explicitly imposed at this level, for it will follow directly from the minimization of a cost functional, as shown in~\eqref{e:optimization1}--\eqref{e:optimization2} below. This avoids some technicalities in Sections~\ref{s:reg} and~\ref{s:optimality}.

\subsection{ Topology optimization}
The primal aim of
topology optimization is to detect the optimal distribution of a
material inside a known box, which is the open set~$\Om$ in our
context. From a mathematical standpoint, such a problem can be
formulated as an optimal control problem, in which the
set~$E\subseteq\Om$ to be filled with the material serves as a
\emph{control parameter}, while the physical properties of the
material act as a constraint through the equilibrium and
constitutive equations. 
The aim of our work is to present and discuss a model for topology
optimization which accounts for possible plastic behaviors  in the
linear kinematic-hardening case.  The target functional  that we
want to minimize is the augmented {\it compliance}
\begin{equation}\label{e:J_sharp}
\J(E, u) \coloneq \int_{E}  \ell (z) \EEE f{\, \cdot\,} u \, \di x + \int_{\Gamma_{N}} g { \, \cdot\, } u \, \di \HH^{n-1} +  \frac{1}{6}  \, \mathrm{Per}( E; \Om) \,.
\end{equation} 
In the latter, the first two terms correspond to the classical
compliance and measure the ability of the body to resist to the
applied load~$\ell(z)f$ and traction $g$. In particular, at equilibrium the first two
terms are nothing but $\int_\Omega \stress\cdot ({\rm E}u -\p)\, {\rm
  d}x$, namely the stored elastic energy.

In \eqref{e:J_sharp}, $E \subseteq \Om$ is assumed to be a set of finite
perimeter in~$\Om$, i.e., the characteristic
function~$\mathbf{1}_{E}$ of~$E$ belongs to the space~$BV(\Om)$ of
functions of bounded variations~\cite{MR1857292}.
The last term in \eqref{e:J_sharp} hence penalizes  sets $E$ with large boundary, limiting the onset
of a very fine microstructure. In this regard, the specific choice of
the constant $1/6$ is merely motivated to ease notations with respect to the  
phase-field 
approximation argument, see  Theorem \ref{t:gamma_conv}. In particular, the analysis is independent from the
specific value of this constant, which can be changed with no
difficulty. Before moving on, let us 
notice that the functional~$\J$
 could be modified to take into account, for instance,
 accumulation of plastic strain~$\p$ or a cost of production depending
 on the size of~$E$. We stay with the specific form above for the
 sake of definiteness.

The {\it sharp-interface} optimization problem  \eqref{min} is
therefore specified as   
\begin{eqnarray}
&& \displaystyle \min_{E \in BV(\Om)} \, \J ( E , u ) \,, \label{e:optimization1.1} \\[1mm]
&& \displaystyle \text{subjected to} \qquad \min_{( u, \strain , \p) \in \A (w)} \, \E (\mathbf{1}_{E}, u, \strain, \p) \,.  \label{e:optimization2.1}
\end{eqnarray}

The forward problem~\eqref{e:optimization2.1} represents the
elastoplastic equilibrium condition, and is formulated at a static,
incremental level. Under this constraint, we look for the optimal 
shape $E\subset \Omega $  by minimizing~$\J$ in~\eqref{e:optimization1.1}.

We now briefly discuss the existence of solutions to~\eqref{e:optimization1.1}--\eqref{e:optimization2.1}.

\begin{proposition}[Existence of optimal sets]
\label{p:1.1}
Let $f \in L^{2}(\Om;\R^{n})$, $g \in L^{2}(\Gamma_{N}; \R^{n})$, and $w \in H^{1}(\Om; \R^{n})$. Then, there exists a solution $E \in BV(\Om)$ to the optimization problem~\eqref{e:optimization1.1}--\eqref{e:optimization2.1}.
\end{proposition}

\begin{proof}
The existence of a solution follows from an application of the Direct Method. Indeed, denoted with $(u_E, \strain_E, \p_E) \in \A(w)$ the unique solution of~\eqref{e:optimization2} for~$E\in BV(\Om)$, we clearly have
\begin{displaymath}
\E(\mathbf{1}_{E}, u_E, \strain_E, \p_E) \leq \E (\mathbf{1}_{E}, w, \e w, 0) \,.
\end{displaymath}
From~\eqref{e:1}--\eqref{e:2} and the previous inequality we deduce that $(u_E, \strain_E, \p_E)$ is bounded in~$H^{1}(\Om ; \R^{n}) \times L^{2}(\Om; \M^{n}_{S}) \times L^{2}(\Om ; \M^{n}_{D})$ uniformly w.r.t.~$E$.

Let~$ E_{k}  \in BV(\Om)$ be  a minimizing sequence. By definition of~$\J$ in~\eqref{e:J_sharp}, we immediately infer that~$E_{k}$ is bounded in~$BV(\Om)$, so that, up to a subsequence, $E_{k} \to E$ in~$L^{1}(\Om)$. Furthermore, also~$(u_{E_{k}}, \strain_{E_{k}}, \p_{E_{k}})$ admits a weak limit~$(u, \strain, \p)$ in $H^{1}(\Om ; \R^{n}) \times L^{2}(\Om; \M^{n}_{S}) \times L^{2}(\Om; \M^{n}_{D})$ with $(u, \strain, \p) \in \A(w)$.

By the lower semicontinuity of~$\E$ we have that the triple~$(u, \strain, \p)$ minimizes $\E(z, \cdot, \cdot, \cdot)$ in~$\A(w)$. Finally, the lower semicontinuity of~$\J$ implies that $E$ is a solution of~\eqref{e:optimization1.1}--\eqref{e:optimization2.1}.
\end{proof}

\begin{remark}
We remark that the forward problem~\eqref{e:optimization2.1} admits a unique solution for every control~$E \in BV(\Om)$, since~$\E(\mathbf{1}_{E}, \cdot, \cdot, \cdot)$ is strictly convex in the state variables. The same holds for the two phase-field optimization problem we consider below, namely~\eqref{e:optimization2} and~\eqref{e:optimization4}.
\end{remark}

In what follows we focus on a phase-field approximation of the
optimization
problem~\eqref{e:optimization1.1}--\eqref{e:optimization2.1}. 
More precisely,  for $\delta>0$ we consider the target functional
\begin{equation}\label{e:J}
\J_{\delta} (z, u) \coloneq \int_{\Om} \ell(z) f {\, \cdot \,} u \, \di x + \int_{\Gamma_{N}} g { \, \cdot \, } u \, \di \HH^{n-1} +  \int_{\Om} \frac{\delta}{2} | \nabla{z} |^{2} + \frac{ z^{2} (1-z)^{2}}{ 2 \delta} \, \di x
\end{equation}
defined for $z\in H^{1}(\Om) \cap L^{\infty} (\Om)$ and $u \in H^{1} (
\Om ; \R^{n} )$. The first two integrals in~\eqref{e:J} are an
approximation of the compliance appearing in~\eqref{e:J_sharp}. The
last integral, instead, is a typical Modica-Mortola
term~\cite{MR866718, MR0445362}, which, according to the value
of~$\delta$, forces~$z$ to take the values~$0$ or~$1$ and penalizes
the transition between material ($\{ z= 1\}$) and void ($\{ z=0 \}$)
regions. In particular, such a term $\Gamma$-converges to the
perimeter functional~$\mathrm{Per} ( \cdot ;\Om) /6$  in~$\mathcal
J$.  

The {\it phase-field} optimization problem  \eqref{min2} is hence
rewritten as  
\begin{eqnarray}
&& \displaystyle \min_{z \in H^{1}(\Om) \cap L^{\infty}(\Om)} \, \J_{\delta} ( z , u ) \,, \label{e:optimization1} \\[1mm]
&& \displaystyle \text{subjected to} \qquad \min_{( u, \strain , \p) \in \A (w)} \, \E (z, u, \strain, \p) \,.  \label{e:optimization2}
\end{eqnarray}

\begin{remark}\label{r:2}
We remark that the usual constraint~$z \in H^{1}(\Om; [0,1])$ has not been explicitly imposed in~\eqref{e:optimization1}. However, if~$z \in H^{1}(\Om) \cap L^{\infty}(\Om)$ is a solution of~\eqref{e:optimization1}--\eqref{e:optimization2}, then also $\overline{z} = 0 \vee z \wedge 1$ is a solution, with the same state configuration. This is due to the assumptions on~$\C$,~$\ha$,~$d$, and~$\ell$ above, which are supposed to be continuously and constantly extended outside the interval~$[0,1]$. 
\end{remark}

 In what follows  we prove the existence of solutions
 for~\eqref{e:optimization1}--\eqref{e:optimization2} and show the
 $\Gamma$-convergence of the phase-field problem to the sharp-interface one 
 as $\delta\to 0$. 

\begin{proposition}[Existence of optimal phase fields]
\label{p:1}
 Let $f \in L^{2}(\Om;\R^{n})$, $g \in L^{2}(\Gamma_{N}; \R^{n})$, and $w \in H^{1}(\Om; \R^{n})$. Then, there exists a solution $z \in H^{1}(\Om;[0,1])$ to the optimization problem~\eqref{e:optimization1}--\eqref{e:optimization2}.
\end{proposition}

\begin{proof}
 The existence of solutions follows from the same argument of Proposition~\ref{p:1.1}. We only have to notice that, by Remark~\ref{r:2}, we may assume, without loss of generality, that a minimizing sequence~$z_{k} \in H^{1}(\Om) \cap L^{\infty}(\Om)$ satisfies $0 \leq z_{k} \leq 1$ in~$\Om$.
%
%
%
\end{proof}

In order to give a compact statement of the $\Gamma$-convergence
result, we introduce \linebreak 
$\G_{\delta} \colon H^{1}(\Om; [0,1]) \times H^{1}(\Om; \R^{n}) \to \R\cup\{+\infty\}$ defined as
\begin{equation}\label{e:G}
\G_{\delta} (z, u) \coloneq \left\{
\begin{array}{ll}
\J_{\delta}(z, u) & \text{if $(u, \strain, \p) \in \A(w)$ is a solution of~\eqref{e:optimization2}}\,,\\[1mm]
+\infty & \text{elsewhere}\,.
\end{array}
\right.
\end{equation}
In a similar way we may define~$\G \colon BV(\Om;[0,1]) \times H^{1}(\Om; \R^{n}) \to \R\cup\{+\infty\}$ by replacing~$\J_{\delta}$ with~$\J$ and~\eqref{e:optimization2} with~\eqref{e:optimization2.1}. 

\begin{remark}
Notice that in the definition~\eqref{e:G} of~$\G_{\delta}$ we have
imposed~$z \in H^{1}(\Om; [0,1])$. As pointed out in Remark~\ref{r:2},
such a requirement is not necessary when we look for minimizers
of~\eqref{e:optimization1}--\eqref{e:optimization2}, as it can be
imposed a posteriori on a solution. However, such a constraint has to
be considered to state a precise $\Gamma$-convergence result, as the
space of sets of finite perimeter is described by characteristic
functions taking values in~$[0,1]$ only.
\end{remark}

\begin{theorem}[Phase-field approximation of the sharp-interface model]
\label{t:gamma_conv}
Let $f \in L^{2}(\Om;\R^{n})$, \linebreak $g \in L^{2}(\Gamma_{N}; \R^{n})$, and $w \in H^{1}(\Om; \R^{n})$. Then, $\G_{\delta}$ $\Gamma$-converges to~$\G$ as $\delta \to 0$ in the $L^{1}$-topology.
\end{theorem}

\begin{proof}
We start by proving the $\Gamma$-liminf inequality. Let
$\delta_{k} \to 0$, $E \in BV(\Om)$, and $z_{k} \in H^{1}(\Om
; [0,1])$ be such that $z_{k} \to \mathbf{1}_{E}$ in
$L^{1}(\Om)$. Let us denote by $(u_{E},
\strain_{E}, \p_{E}), (u_{z_{k}}, \strain_{z_{k}}, \p_{z_{k}}) \in
\A(w)$ the corresponding state variables, solutions
of~\eqref{e:optimization2.1} and~\eqref{e:optimization2},
respectively. Following the steps of the proof of
Proposition~\ref{p:1.1}, we get that $(u_{z_{k}}, \strain_{z_{k}},
\p_{z_{k}}) \rightharpoonup  (u_{E}, \strain_{E}, \p_{E})$ weakly in $H^{1}(\Om;\R^{n}) \times L^{2}(\Om; \M^{n}_{S}) \times L^{2} (\Om; \M^{n}_{D})$. Furthermore, from the $\Gamma$-convergence of the Modica-Mortola term to the perimeter functional~\cite{MR866718, MR0445362} we deduce
\begin{displaymath}
\G(E, u_{E}) \leq \liminf_{k\to\infty} \, \G_{\delta_{k}} ( z_{k}, u_{z_{k}})\,.
\end{displaymath}

As for the $\Gamma$-limsup inequality, for every $E \in BV(\Om)$ we
consider the recovery sequence $z_{k} \in H^{1}(\Om; [0,1])$
constructed in~\cite{MR866718, MR0445362} and such that $z_{k} \to
\mathbf{1}_{E}$ in~$L^{1}(\Om)$. Again, let $(u_{E}, \strain_{E},
\p_{E}),\, (u_{z_{k}}$, $\strain_{z_{k}}, \p_{z_{k}}) \in \A(w)$ be the corresponding state variables. It is easy to check that $(u_{z_{k}},
\strain_{z_{k}}, \p_{z_{k}}) \rightharpoonup   (u_{E}, \strain_{E}, \p_{E})$
weakly in $H^{1}(\Om;\R^{n}) \times L^{2}(\Om; \M^{n}_{S}) \times L^{2} (\Om; \M^{n}_{D})$ and
\begin{equation}\label{e:s1}
\limsup_{k \to \infty} \, \G_{\delta_{k}} ( z_{k}, u_{z_{k}}) \leq \G(E, u_{E})\,.
\end{equation}
This concludes the proof of the theorem.
\end{proof}

\begin{corollary}[Convergence of optimal controls]\label{cor}
Under the assumptions of Theorem~\emph{\ref{t:gamma_conv}}, every sequence~$z_{\delta} \in H^{1}(\Om; [0,1])$ of solutions to~\eqref{e:optimization1}--\eqref{e:optimization2} admits, up to a subsequence, a limit~$\mathbf{1}_{E}$ in~$L^{1}(\Om)$, where $E \in BV(\Om)$ is a solution to~\eqref{e:optimization1.1}--\eqref{e:optimization2.1}.
\end{corollary}

\begin{proof}
The thesis follows by the $\Gamma$-convergence shown in Theorem~\ref{t:gamma_conv}.
\end{proof}

In the rest of the paper we focus on the first-order optimality conditions for~\eqref{e:optimization1}--\eqref{e:optimization2}. In this respect, we notice that the presence of the positively $1$-homogeneous plastic dissipation prevents us from deducing suitable optimality conditions by operating directly on~\eqref{e:optimization1}--\eqref{e:optimization2}. For this reason, we introduce a regularization of the energy functional~\eqref{e:energy} depending on a parameter~$\gamma>0$. Namely, for every  $z \in L^{\infty}(\Om)$,  every $(u, \strain, \p) \in \A(w)$, and every $\gamma \in (0, +\infty)$, we define
\begin{equation}\label{e:energygamma}
\begin{split}
\E_{\gamma} (z, u, \strain, \p) \coloneq & \  \frac{1}{2} \int_{\Om} \C(z) \strain {\, \cdot \,} \strain \, \di x + \frac{1}{2} \int_{\Om} \ha(z) \p{\,\cdot \,} \p \, \di x + \int_{\Om} d(z) \left( \sqrt{ | \p |^{2} + \frac{1}{\gamma^{2}}} - \frac{1}{\gamma} \right) \, \di x
\\
&
 - \int_{\Om}  \ell(z)  f {\, \cdot \, } u \, \di x - \int_{\Gamma_{N}} g {\, \cdot\,} u \, \di \HH^{n-1} \,.
 \end{split}
\end{equation} 
For later use, we set
\begin{displaymath}
h_{\gamma} (\q) \coloneq  \sqrt{|\boldsymbol{\q}|^{2} + \frac{1}{\gamma^{2}}} - \frac{1}{\gamma} \qquad \text{for every $\boldsymbol{\q} \in \M^{n}_{D}$},
\end{displaymath}
and notice that~$h_{\gamma} \in C^{\infty} (\M^{n}_{D})$ is convex and such that 
\begin{eqnarray}
&& \displaystyle h_{\gamma}(\mathrm{0}) = 0 \,, \qquad |\boldsymbol{\q}| - \frac{1}{\gamma} \leq h_{\gamma} (\boldsymbol{\q}) \leq |\boldsymbol{\q} | \,, \label{e:h1} \\[1mm]
&& \displaystyle \vphantom{\frac12}| h_{\gamma} (\boldsymbol{\q}_{1})  - h_{\gamma} (\boldsymbol{\q}_{2})| \leq  | \boldsymbol{\q}_{1} - \boldsymbol{\q}_{2}|  \,, \label{e:h2} \\ [1mm]
&& \displaystyle \vphantom{\frac12}| \nabla_{\boldsymbol{\q}} h_{\gamma} (\boldsymbol{\q}_{1}) - \nabla_{\boldsymbol{\q}} h_{\gamma}( \boldsymbol{\q}) | \leq 2 \gamma | \boldsymbol{\q}_{1} - \boldsymbol{\q}_{2} |  \label{e:h3}
\end{eqnarray}
for every $\gamma \in (0,+\infty)$ and every $\boldsymbol{\q}, \boldsymbol{\q}_{1}, \boldsymbol{\q}_{2} \in \M^{n}_{D}$.

We now consider the regularized optimization problem
\begin{eqnarray}
&& \displaystyle \min_{z \in H^{1}(\Om) \cap L^{\infty}(\Om)} \, \J_{\delta}(z, u) \,, \label{e:optimization3} \\[1mm]
&& \displaystyle \text{subjected to} \qquad \min_{( u, \epsilon , \p) \in \A(w)} \, \E_{\gamma} (z, u, \strain, \p) \,.  \label{e:optimization4}
\end{eqnarray}
We notice that the minimization problem~\eqref{e:optimization3}--\eqref{e:optimization4} is expected to be regular w.r.t.~$z$ (see Theorem~\ref{t:regoptim}) as the regularization of the plastic dissipation
\begin{displaymath}
\int_{\Om} d(z) h_{\gamma}(\p) \, \di x 
\end{displaymath}
is differentiable w.r.t.~$\p \in L^{2}(\Om; \M^{n}_{D})$.

In the following proposition we state the existence of solutions of~\eqref{e:optimization3}--\eqref{e:optimization4} and the convergence as~$\gamma\to +\infty$ to solutions of~\eqref{e:optimization1}--\eqref{e:optimization2}.

\begin{proposition}[Existence and convergence of regularized optimal controls]
\label{p:2}  
Assume to be \linebreak given $f\in L^{2}(\Om; \R^{n})$, $g \in L^{2}(\Gamma_{N}; \R^{n})$, and $w \in H^{1}(\Om ; \R^{n})$. Then, for every $\gamma \in (0,+\infty)$ there exists a solution $z_{\gamma} \in H^{1}( \Om ; [0,1])$ to~\eqref{e:optimization3}--\eqref{e:optimization4}.

Moreover, every sequence~$z_{\gamma}$ of solutions
of~\eqref{e:optimization3}--\eqref{e:optimization4} admits a
subsequence which is weakly convergent to some~$z \in H^{1}(\Om ;
[0,1])$ solving~\eqref{e:optimization1}--\eqref{e:optimization2}, and
the corresponding state configuration~$(u_{\gamma}, \strain_{\gamma},
\p_{\gamma})$ converges in~$H^{1}(\Om; \R^{n}) \times L^{2}(\Om;
\M^{n}_{S}) \times L^{2}(\Om; \M^{n}_{D})$ to~$(u, \strain, \p)$ solving~\eqref{e:optimization2}.
\end{proposition}

\begin{proof}
The existence of solutions of~\eqref{e:optimization3}--\eqref{e:optimization4} can be shown as in Propositions~\ref{p:1.1} and~\ref{p:1}.

Let us prove the second part of the statement. Let $z_{\gamma} \in H^{1}(\Om; [0,1])$ be a sequence of solutions of~\eqref{e:optimization3}--\eqref{e:optimization4} for $\gamma \in (0,+\infty)$, and let us denote with~$(u_{\gamma}, \strain_{\gamma}, \p_{\gamma})$ the corresponding solution of the forward problem~\eqref{e:optimization4}. Then, we have that
\begin{displaymath}
\E_{\gamma} (z_{\gamma}, u_{\gamma}, \strain_{\gamma}, \p_{\gamma}) \leq \E_{\gamma} (z_{\gamma}, w, \e w , 0) \leq \beta_{\C} \| w\|_{H^{1}}^{2} + ( \| f\|_{2} + \| g \|_{2})  \|w\|_{H^{1}} \,.
\end{displaymath} 
Hence, $(u_{\gamma}, \strain_{\gamma}, \p_{\gamma})$ is bounded in~$H^{1}(\Om; \R^{n}) \times L^{2}(\Om ; \M^{n}_{S}) \times L^{2}(\Om; \M^{n}_{D})$ independently of~$\gamma$ and admits, up to a subsequence, a weak limit $(u, \strain, \p) \in \A(w)$. By optimality of~$z_{\gamma}$ we also infer that~$z_{\gamma}$ is bounded in~$H^{1}(\Om ; [0,1])$ and, up to a further subsequence,~$z_{\gamma} \rightharpoonup z$ weakly in~$H^{1}(\Om; [0,1])$.

Let us show that~$(u, \strain, \p) \in \A(w)$ is the minimizer of~$\E(z, \cdot, \cdot, \cdot)$ in~$\A(w)$. Thanks to~\eqref{e:h1} we have that for every $(v, \eeta, \qq) \in \A(w)$ it holds
\begin{equation}\label{e:4}
\E ( z_{\gamma}, u_{\gamma}, \strain_{\gamma}, \p_{\gamma}) - \frac{M_d}{\gamma} |\Om|  \leq \E_{\gamma}( z_{\gamma}, u_{\gamma}, \strain_{\gamma}, \p_{\gamma}) \leq \E_{\gamma}( z_{\gamma}, v, \eeta, \qq) \leq \E ( z_{\gamma}, v, \eeta , \qq) \,,
\end{equation}
where we have set $M_{d} \coloneq \max \{ d( z ) : \,  z  \in[0,1]\}$. Passing to the liminf as~$\gamma \to +\infty$ on both sides of~\eqref{e:4} we get
\begin{displaymath}
\E(z, u, \strain, \p) \leq \E(z, v, \eeta, \qq) \qquad \text{for every $(v, \eeta, \qq) \in \A(w)$}\,.
\end{displaymath}
Hence, $(u, \strain , \p) \in \A(w)$ is a minimizer of~$\E(z, \cdot, \cdot, \cdot)$ in~$\A(w)$. Moreover, from the argument above we deduce that $\E(z_{\gamma}, u_{\gamma}, \strain_{\gamma}, \p_{\gamma})$ converges to $\E(z, u, \strain, \p)$, which implies the strong convergence of $(u_{\gamma}, \strain_{\gamma}, \p_{\gamma})$ to $(u, \strain, \p)$ in~$H^{1}(\Om ; \R^{n}) \times L^{2}(\Om; \M^{n}_{S}) \times L^{2}(\Om; \M^{n}_{D})$.
\end{proof}

\begin{remark}
We remark that the solutions of~\eqref{e:optimization1}--\eqref{e:optimization2} and~\eqref{e:optimization3}--\eqref{e:optimization4} are not unique, since the functional~$\J_{\delta}$ is not convex. In particular, we can not ensure that all the solutions of~\eqref{e:optimization1}--\eqref{e:optimization2} can be approximated by solutions of~\eqref{e:optimization3}--\eqref{e:optimization4}. 

Since in the next sections we are going to deduce the optimality
conditions for~\eqref{e:optimization1}--\eqref{e:optimization2} as
limit of those for~\eqref{e:optimization3}--\eqref{e:optimization4}, these will be valid only for a subclass of solutions of~\eqref{e:optimization1}--\eqref{e:optimization2}. Clearly, if~\eqref{e:optimization1}--\eqref{e:optimization2} admits a unique solution, in view of Proposition~\ref{p:2} it can be approximated by solutions of~\eqref{e:optimization3}--\eqref{e:optimization4}, and its first order optimality conditions follows from Theorems~\ref{t:regoptim} and~\ref{t:lim_optim}.
\end{remark}


\section{Optimality of the regularized problem}
\label{s:reg}

This section is devoted to the computation of the first-order optimality conditions for the regularized problem~\eqref{e:optimization3}--\eqref{e:optimization4}. In particular, we follow here the main lines of~\cite{delosReyes}.

In order to state the main result of this section, we need some additional notation. For $\gamma \in (0,+\infty)$, $z, \varphi \in L^{\infty}(\Om )$, and $( u , \strain, \p), (v, \eeta, \qq) \in H^{1}(\Om ; \R^{n}) \times L^{2}(\Om ; \M^{n}_{S}) \times L^{2}(\Om ; \M^{n}_{D})$, we set
 \begin{align}\label{e:F}
\F^{(u, \strain, \p)}_{\gamma} (v, \eeta, \qq) \coloneq & \ \frac{1}{2} \int_{\Om} \C(z) \eeta{\, \cdot\, }\eeta \, \di x + \frac{1}{2} \int_{\Om} \ha(z) \qq{ \, \cdot \,} \qq \, \di x + \int_{\Om} ( \C'(z) \varphi) \strain {\, \cdot\, } \eeta \, \di x 
\\
&
+ \int_{\Om} (\ha'(z) \varphi) \p {\, \cdot\,} \qq \, \di x + \int_{\Om} \varphi \, d'(z)  \nabla_{\q} h_{\gamma}(\p)  {\, \cdot\,} \qq \, \di x \nonumber
\\
&
+ \int_{\Om} d(z) \big( \nabla^{2}_{\q} h_{\gamma} (\p) \qq\big){ \, \cdot \,} \qq \, \di x  - \int_{\Om} \varphi\, \ell'(z)  \, f {\,\cdot\,} v \, \di x\,. \nonumber 
\end{align}  

\begin{theorem}[Differentiability of the control-to-state map]
\label{t:regoptim}
Let $p\in (2, +\infty)$, $f\in L^{p}(\Om; \R^{n})$, $g \in L^{p}(\Gamma_{N}; \R^{n})$, $w\in W^{1, p}( \Om ; \R^{n})$, and $\gamma \in (0,+\infty)$. Then, the control-to-state operator $S_{\gamma} \colon L^{\infty} ( \Om ) \to \A(w)$ defined as
\begin{displaymath}
S_{\gamma} (z) \coloneq \argmin \, \{ \E_{\gamma} (z, u, \strain, \p) : \,( u, \strain , \p) \in \A(w) \}
\end{displaymath} 
is Frech\'et differentiable. If $(u_{\gamma}, \strain_{\gamma}, \p_{\gamma}) = S_{\gamma}(z)$, the derivative of~$S_{\gamma}$ in the direction~$\varphi \in L^{\infty}(\Om)$ is given by 
\begin{equation}\label{e:der}
S'_{\gamma}(z) [\varphi] =  \argmin \, \{ \F^{(u_{\gamma}, \strain_{\gamma}, \p_{\gamma})}_{\gamma} (v, \eeta, \qq): \, (v, \eeta, \qq) \in \A(0) \} \,. 
\end{equation}
\end{theorem}

\begin{remark}
We notice that the minimization problem~\eqref{e:der} admits a unique solution, since the functional~$\F^{(u_{\gamma}, \strain_{\gamma}, \p_{\gamma})}_{\gamma}$ is strictly convex as a consequence of assumptions~\eqref{e:1}--\eqref{e:2}, of the positivity of~$d$, and of the convexity of~$h_{\gamma}$.
\end{remark}

\begin{remark}
We remark that the additional $p$-integrability of the   density of the applied body force~$f$\EEE, of the applied traction $g$, and of the boundary datum~$w$ is necessary to prove the differentiability of the control-to-state operator~$S_{\gamma}$, while they are not needed to show existence of solutions of the optimality problems~\eqref{e:optimization1.1}--\eqref{e:optimization2.1},\eqref{e:optimization1}--\eqref{e:optimization2}, and~\eqref{e:optimization3}--\eqref{e:optimization4}.
\end{remark}

From Theorem~\ref{t:regoptim} we will infer the first-order optimality conditions for the regularized problem~\eqref{e:optimization3}--\eqref{e:optimization4}.

\begin{corollary}[First-order conditions for the regularized problem]
\label{c:regoptim}
In the framework of Theorem~\ref{t:regoptim}, assume that $z_{\gamma} \in H^{1} (\Om; [0,1])$ is a solution of~\eqref{e:optimization3}--\eqref{e:optimization4} with corresponding displacement~$( u_{\gamma}, \strain_{\gamma}, \p_{\gamma} ) \in \A(w)$. Then, there exists~$ ( \overline{u}_{\gamma}, \overline{\strain}_{\gamma}, \overline{\p}_{\gamma} ) \in \A(0)$ such that for every $(v, \eeta, \qq) \in \A(0)$ and every $\varphi \in L^{\infty}(\Om) \cap H^{1}(\Om)$
 \begin{align}
 & \int_{\Om} \C(z_{\gamma})  \overline{\strain}_{\gamma} {\, \cdot\, }\eeta \, \di x 
+ \int_{\Om} \ha(z_{\gamma}) \overline{\p}_{\gamma} { \, \cdot \,} \qq \, \di x \label{e:optreg2}
\\
& \qquad\qquad
+ \int_{\Om} d (z_{\gamma} ) \big( \nabla^{2}_{\q} h_{\gamma} (\p_{\gamma}) \overline{\p}_{\gamma} \big){ \, \cdot \,} \qq \, \di x  - \int_{\Om} \ell (z) \, f{\, \cdot\,} v \, \di x - \int_{\Gamma_{N}} g {\, \cdot\,} v \, \di \HH^{n-1} = 0 \,,\nonumber 
  \\[1mm] 
&  \int_{\Om}  \varphi \, \ell'(z_{\gamma} ) f {\, \cdot\,}( \overline{ u}_{\gamma} + u_{\gamma}) \, \di x  -\int_{\Om} \big( \C'(z_{\gamma}) \varphi\big) \overline{\strain}_{\gamma}{\, \cdot \, } \strain_{\gamma} \, \di x \label{e:optreg}
 \\
 & \qquad\qquad
 - \int_{\Om} \big( \ha'(z_{\gamma}) \varphi \big) \overline{\p}_{\gamma} {\, \cdot \,} \p_{\gamma} \, \di x 
   - \int_{\Om} \varphi d' ( z_{\gamma}) \nabla_{\q} h_{\gamma} ( \p_{\gamma} ) {\, \cdot\,} \overline{\p}_{\gamma} \, \di x  \nonumber
    \\
 &
 \qquad\qquad
 + \int_{\Om}  \delta \nabla{z}_{\gamma} {\, \cdot \, } \nabla{\varphi}  + \frac{1}{\delta}\,\varphi  (z_{\gamma} (1 - z_{\gamma})^{2} - z_{\gamma}^{2} (1 - z_{\gamma})) \, \di x  = 0 \,. \nonumber
\end{align}
\end{corollary}

In order to prove the Frech\'et differentiability of the control-to-state operator~$S_{\gamma}$ stated in Theorem~\ref{t:regoptim} we first need to investigate the integrability and continuity properties of~$S_{\gamma} (z)$ for~$z\in L^{\infty}(\Om)$. This is the subject of the following three lemmas.

\begin{lemma}\label{l:1}
For every $ z  \in \R$ and every $\gamma \in (0, + \infty)$ let $F_{ z , \gamma} \colon \M^{n}_{D} \to \M^{n}_{D}$ be the map defined as
\begin{equation}\label{e:Fgamma}
F_{ z , \gamma} ( \q) \coloneq \C( z ) \q + \ha ( z ) \q + d(z) \nabla_{\q} h_{\gamma} (\q)  \qquad \text{for every $\q \in \M^{n}_{D}$}\,.
\end{equation}
Then, there exist three constants~$C_{1}, C_{2}, C_{3} > 0$ independent of~$\gamma$ and~$ z $ and a constant~$C_{\gamma}>0$ (depending only on~$\gamma$) such that the following inequalities hold:
\begin{align}
& | F_{ z , \gamma} (\q_{1}) - F_{ z , \gamma} (\q_{2}) | \leq C_{\gamma} | \q_{1} - \q_{2} | \,, \label{e:11} \\[1mm]
&  \big( F_{ z , \gamma} (\q_{1}) - F_{ z , \gamma} (\q_{2}) \big) \cdot (\q_{1} - \q_{2}) \geq C_{1} | \q_{1} - \q_{2} |^{2} . \label{e:12} \\[1mm]
& (F_{ z _{1}, \gamma} (\q_{1}) - F_{ z _{2} , \gamma} (\q_{2} ) ) \cdot (\q_{1} - \q_{2} ) \label{e:9}
\\
&
\quad \geq  C_{1} | \q_{1} - \q_{2} |^{2} - C_{2} | \q_{2} | | z _{1} -  z _{2} | |\q_{1} - \q_{2} | 
- C_{3} |  z _{1} -  z _{2} | | \q_{1} - \q_{2} | \nonumber \,, 
\end{align}
for every $ z ,  z _{1},  z _{2} \in \R$ and every $\q_{1}, \q_{2} \in \M^{n}_{D}$.

Moreover,~$F_{ z , \gamma}$ is invertible for every $ z  \in \R$ and every~$\gamma \in (0,+\infty)$, and its inverse satisfies
\begin{equation} \label{e:15}
| F_{ z , \gamma}^{-1} ( \q_{1}) - F_{ z , \gamma}^{-1} (\q_{2}) | \leq \widetilde{C} | \q_{1} - \q_{2} |  
\end{equation}
for every $\q_{1}, \q_{2} \in \M^{n}_{D}$, for a positive constant~$\widetilde{C}$ independent of~$\gamma$ and~$ z $.
\end{lemma}

\begin{proof}
Inequality~\eqref{e:11} follows from assumptions~\eqref{e:1}--\eqref{e:2}, from the Lipschitz continuity of~$d$, and from inequality~\eqref{e:h3}. In particular,~$C_{\gamma}$ depends on~$\gamma$ because the Lipschitz constant in~\eqref{e:h3} degenerates with~$\gamma$. Inequality~\eqref{e:12} is a consequence of~\eqref{e:1}--\eqref{e:2}, of the convexity of~$h_{\gamma}$, and of the sign of~$d$. The constant~$C_{1}$ is independent of~$\gamma \in (0,+\infty)$ as~\eqref{e:1}--\eqref{e:2} are. From~\eqref{e:11}--\eqref{e:12} we infer that~$F_{ z , \gamma}$ is invertible, with inverse~$F_{ z , \gamma}^{-1} \colon \M^{n}_{D} \to \M^{n}_{D}$ satisfying~\eqref{e:15}.

Let us now show~\eqref{e:9}. By definition of~$F_{ z , \gamma}$, for every $ z _{1},  z _{2} \in \R$ and every~$\q_{1}, \q_{2} \in \M^{n}_{D}$ we have 
\begin{align}\label{e:10}
& ( F_{ z _{1}, \gamma}  (\q_{1}) - F_{ z _{2} , \gamma} (\q_{2} ) ) \cdot (\q_{1} - \q_{2} ) 
\\
&
= \big( \C ( z _{1}) \q_{1} + \ha( z _{1}) \q_{1} + d( z _{1} ) \nabla_{\q}h_{\gamma} (\q_{1}) - \C( z _{2}) \q_{2} \big) \cdot ( \q_{1} - \q_{2} ) \nonumber
\\
&
\qquad + \big( - \ha( z _{2}) \q_{2} - d( z _{2} ) \nabla_{\q} h_{\gamma} (\q_{2}) \big) \cdot ( \q_{1} - \q_{2} ) \nonumber 
\\
&
= \big( F_{ z _{1}, \gamma} (\q_{1}) - F_{ z _{1}, \gamma}( \q_{2} ) \big) \cdot (\q_{1} - \q_{2} ) + \big( (\C ( z _{1}) - \C( z _{2})) \q_{2} + (\ha( z _{1}) - \ha( z _{2})) \q_{2} \nonumber
\\
& \qquad + ( d( z _{1}) - d( z _{2}) )\nabla_{\q} h_{\gamma}(\q_{2}) \big) \cdot (\q_{1} - \q_{2}) \,. \nonumber
\end{align}
Inequality~\eqref{e:9} can be deduced from~\eqref{e:10} by taking into account the assumptions~\eqref{e:1}--\eqref{e:2}, the regularity of~$\C$,~$\ha$, and~$d$ w.r.t.~$ z $, inequality~\eqref{e:12}, and the bound $ | \nabla_{\q} h_{\gamma}(\q_{2}) | \leq 1$. 
\end{proof}

\begin{lemma}\label{l:2}
For every $\gamma \in (0,+\infty)$ and every $ z  \in \R$ let the map $b_{ z , \gamma} \colon \M^{n}_{S} \to \M^{n}_{S}$ be defined as
\begin{equation}\label{e:bgamma}
b_{ z , \gamma}( \boldsymbol{\e}) \coloneq \C ( z ) \big ( \boldsymbol{\e}- F^{-1}_{ z , \gamma} ( \Pi_{\M^{n}_{D}} ( \C(  z ) \boldsymbol{\e} ))\big) \qquad \text{for every $\boldsymbol{\e} \in \M^{n}_{S}$},
\end{equation}
where $\Pi_{\M^{n}_{D}} \colon \M^{n} \to \M^{n}_{D}$ denotes the projection operator on~$\M^{n}_{D}$. Then, there exist two positive constants~$c_{1}, c_{2}$ such that for every~$\gamma \in (0, +\infty)$, every $ z  \in \R$, and every $\boldsymbol{\e}_{1}, \boldsymbol{\e}_{2} \in \M^{n}_{S}$
\begin{eqnarray}
&& \displaystyle | b_{ z , \gamma} ( \boldsymbol{\e}_{1}) - b_{ z , \gamma} (\boldsymbol{\e}_{2} ) | \leq c_{1} |\boldsymbol{\e} _{1} - \boldsymbol{\e}_{2} | \,, \label{e:13}\\[1mm]
&& \displaystyle ( b_{ z , \gamma} (\boldsymbol{\e}_{1}) - b_{ z ,\gamma} (\boldsymbol{\e}_{2}) ) \cdot (\boldsymbol{\e}_{1} - \boldsymbol{\e}_{2} ) \geq c_{2} | \boldsymbol{\e}_{1} - \boldsymbol{\e}_{2} |^{2} \,. \label{e:14}
\end{eqnarray}
\end{lemma}

\begin{proof}
Property~\eqref{e:13} is a direct consequence of~\eqref{e:15}. Let us prove~\eqref{e:14}, instead. Following the ideas of~\cite[Section~5]{delosReyes}, let us define for $ z  \in \R$, $\boldsymbol{\e} \in \M^{n}_{S}$, and~$\q \in \M^{n}_{D}$ the auxiliary function
\begin{displaymath}
G_{ z ,\gamma} (\e, \q) \coloneq \left(  \begin{array}{cc}
\C( z )( \boldsymbol{\e} - \q) \\ [1mm]
F_{ z ,\gamma} (\q) - \Pi_{\M^{n}_{D}} ( \C( z ) \boldsymbol{\e}) 
\end{array}
\right).
\end{displaymath} 
Then, for every $ z  \in \R$, every $\boldsymbol{\e}_{1}, \boldsymbol{\e}_{2} \in \M^{n}_{S}$, and every $\q_{1}, \q_{2} \in \M^{n}_{D}$ we have
\begin{align}\label{e:20}
& \big(   G_{ z ,\gamma}   (\boldsymbol{\e}_{1}, \q_{1}) - G_{ z , \gamma} (\boldsymbol{\e}_{2}, \q_{2}) \big) \cdot \left( \begin{array}{cc} \boldsymbol{\e}_{1} - \boldsymbol{\e}_{2} \\ \q_{1} - \q_{2} \end{array} \right)
\\
&
\quad = \C( z ) \big( (\boldsymbol{\e}_{1} - \q_{1}) - (\boldsymbol{\e}_{2} - \q_{2}) \big) \cdot (\boldsymbol{\e}_{1} - \boldsymbol{\e}_{2}) \nonumber
\\
&
 \qquad + \Pi_{\M^{n}_{D}} \big( \C( z ) (\q_{1} - \q_{2}) + \ha( z ) (\q_{1} - \q_{2}) \big) \cdot (\q_{1} - \q_{2}) \nonumber
 \\
 &
 \qquad+  \big(d( z ) \nabla_{\q} h_{\gamma} (\q_{1}) - d( z )\nabla_{\q} h_{\gamma} (\q_{2})\big) \cdot (\q_{1} - \q_{2})\nonumber
 \\
 &
 \qquad  - \Pi_{\M^{n}_{D}} \big( \C(z) ( \boldsymbol\e_{1} - \boldsymbol\e_{2}) \big) \cdot ( \q_{1} - \q_{2}) \,.  \nonumber
\end{align}
By the convexity of~$h_{\gamma}$ and the assumptions~\eqref{e:1}--\eqref{e:2} we deduce from~\eqref{e:20} that there exists~$C>0$ independent of~$\gamma$,~$ z $,~$\boldsymbol{\e}_{1}, \boldsymbol{\e}_{2}$, and~$\q_{1}, \q_{2}$ such that
\begin{equation}\label{e:22}
\big( G_{ z ,\gamma} (\boldsymbol{\e}_{1}, \q_{1}) - G_{ z , \gamma} (\boldsymbol{\e}_{2}, \q_{2}) \big) \cdot \left( \begin{array}{cc} \boldsymbol{\e}_{1} - \boldsymbol{\e}_{2} \\ \q_{1} - \q_{2} \end{array} \right) \geq C( |\boldsymbol{\e}_{1} - \boldsymbol{\e}_{2} |^{2} + | \q_{1} - \q_{2}|^{2})\,.
\end{equation}
By rewriting~\eqref{e:22} with the particular choice $\q_{i} = F_{ z , \gamma}^{-1}( \Pi_{\M^{n}_{D}} (\C( z ) \boldsymbol{\e}_{i}))$ for $i = 1, 2$ we get~\eqref{e:14}, as $\C( z ) (\boldsymbol{\e}_{i} - \q_{i}) = b_{ z , \gamma} (\boldsymbol{\e}_{i})$ and  $F_{ z , \gamma}(\q_{i}) - \Pi_{\M^{n}_{D}} ( \C ( z ) \boldsymbol{\e}_{i}) = 0$. 
\end{proof}

\begin{lemma}[Bounds on the regularized control-to-state map]
\label{l:integrability}
Let $p\in (2, +\infty)$, $f\in L^{p}(\Om; \R^{n})$, $g \in L^{p}(\Gamma_{N}; \R^{n})$, and $w\in W^{1, p}( \Om ; \R^{n})$. Then, there exist $\tilde{p} \in (2, p)$ and a positive constant~$C>0$ such that for every   $q \in (2, \tilde{p}]$\EEE, every $z, z_{1}, z_{2} \in L^{\infty}(\Om)$, and every $\gamma \in (0,+\infty)$, the following holds:
\begin{eqnarray}
&&\displaystyle \| S_{\gamma}(z) \|_{W^{1, \tilde{p}}\times L^{\tilde{p}} \times L^{\tilde{p}}} \leq C (\| f \|_{p}  + \| g \|_{p} + \| w \|_{W^{1,p}}) \,, \label{e:integrability1} \\[1mm]
&& \displaystyle   \| S_{\gamma}(z_{1}) - S_{\gamma} (z_{2}) \|_{W^{1, q}\times L^{q} \times L^{q}} \leq C   \| z_{1} - z_{2} \|_{\infty} \EEE ( \| f \|_{p} + \| g \|_{p}  + \| w \|_{W^{1, p}} + 1) \,. \label{e:integrability2} 
\end{eqnarray}
\end{lemma}

\begin{proof}
The proof of~\eqref{e:integrability1} and~\eqref{e:integrability2} follows from an application of~\cite[Theorem~1.1]{Herzog}. To apply such result, we first have to recast the Euler-Lagrange equations associated to the minimization problem~\eqref{e:optimization4} in terms of the sole displacement variable~$u$.

Let us fix $\gamma>0$ and~$z \in L^{\infty}(\Om)$. For simplicity of notation, let $(u, \strain, \p) = S_{\gamma}(z)$ and $(u_{i}, \strain_{i}, \p_{i}) = S_{\gamma} (z_{i})$, $i=1, 2$. From the minimization problem~\eqref{e:optimization4} we deduce that the following Euler-Lagrange equation holds: for every~$ (v, \eeta, \qq) \in \A(0)$
\begin{align}\label{e:EL}
&\int_{\Om} \C(z) (\e u - \p){\, \cdot\,} \eeta \, \di x  + \int_{\Om} \ha(z) \p{\, \cdot\,} \qq \, \di x + \int_{\Om} d(z) \nabla_{\q} h_{\gamma} ( \p ){\, \cdot\,} \qq \, \di x
\\
& \qquad
 - \int_{\Om}  \ell(z)  f {\, \cdot\, } v \, \di x - \int_{\Gamma_{N}} g { \, \cdot\,} v \, \HH^{n-1} = 0 \,, \nonumber
\end{align}
where $\e u$ denotes the symmetric part of the gradient of~$u$. By testing~\eqref{e:EL} with $(0, \eeta, -\eeta) \in \A(0)$ for $\eeta \in L^{2}(\Om; \M^{n}_{D})$ we get that
\begin{equation}\label{e:7}
\C ( z) \p + \ha(z) \p + d(z) \nabla_{\q} h_{\gamma} ( \p ) = \Pi_{\M^{n}_{D}}(  \C (z) \e u) \qquad \text{a.e.~in~$\Om$}\,.
\end{equation}
In view of the definition~\eqref{e:Fgamma} of~$F_{ z , \gamma}$, we  have $F_{z(x), \gamma} (\p (x) ) = \Pi_{\M^{n}_{D} }\big ( \C(z(x)) \e u (x) \big)$ and $\p(x) = F_{z(x), \gamma}^{-1} \big (\Pi_{\M^{n}_{D}} \big ( \C(z(x)) \e u (x) \big)\big)$ for a.e.~$x \in \Om$.

Recalling definition~\eqref{e:bgamma}, we define for $x \in \Om$
\begin{displaymath}
b_{z, \gamma} (x, \boldsymbol{\e}) \coloneq b_{z(x), \gamma} (\boldsymbol{\e}) =  \C (z(x)) \big (\boldsymbol{\e} - F^{-1}_{z(x), \gamma} (\Pi_{\M^{n}_{D}} (\C(z(x)) \boldsymbol{\e} ) \big) \qquad \text{for~$x \in \Om$ and~$\boldsymbol{\e} \in \M^{n}_{S}$}.
\end{displaymath}
From now on, when not explicitly needed, we drop the dependence on the spatial variable~$x \in \Om$ in the definition of~$F_{z, \gamma}^{-1}$, since all the arguments discussed below are valid uniformly in~$\Om$. We rewrite the Euler-Lagrange equation~\eqref{e:EL} in terms of the sole displacement~$u$ and for test functions of the form $(\psi, \e \psi, 0) \in \A(0)$ for $\psi \in H^{1} (\Om; \R^{n})$ with $\psi = 0$ on~$\Gamma_{D}$:
\begin{equation}\label{e:EL2}
\int_{\Om} b_{z, \gamma} (x, \e u) {\, \cdot\,} \e \psi \, \di x = \int_{\Om}  \ell(z)  f {\, \cdot\, } \psi \, \di x + \int_{\Gamma_{N}} g {\, \cdot\,} \psi \, \di \HH^{n-1} \,.
\end{equation}

In view of~\eqref{e:13}--\eqref{e:14}, the nonlinear operator~$B_{z, \gamma} \colon W^{1, p}( \Om; \R^{n}) \to W^{-1, p}( \Om; \R^{n})$ defined as $B_{z, \gamma} (u) \coloneq b_{z, \gamma} (x, \e u)$ satisfies the hypotheses of~\cite[Theorem~1.1]{Herzog}. Since $\Om \cup \Gamma_{N}$ is Gr\"oger regular, $ p \in (2, +\infty)$, $f \in L^{p} (\Om ; \R^{n})$, $g \in L^{p}(\Gamma_{N}; \R^{n})$, and $w \in W^{1, p}(\Om; \R^{n})$, we infer from~\cite[Theorem~1.1]{Herzog} applied to equation~\eqref{e:EL2} that there exist~$\tilde{p} \in (2, p)$ and a constant~$C>0$ such that
 \begin{equation}\label{e:26.1}
 \| u \|_{W^{1, q}} \leq C ( \| f \|_{p} + \| g \|_{p} + \| w \|_{W^{1, p}})
 \end{equation}
 for every $q \in (2, \tilde{p}]$. In particular, $C$ is independent of~$z \in L^{\infty}(\Om)$, of~$\gamma \in (0,+\infty)$, and of~$q \in (2, \tilde{p}]$. Inequality~\eqref{e:integrability1}  can be deduced by combining~\eqref{e:12} and~\eqref{e:26.1}. Indeed, we have that
 \begin{displaymath}
 \| F^{-1}_{z, \gamma} (\Pi_{\M^{n}_{D}} (\C(z) \e u)) \|_{q} \leq C \| \Pi_{\M^{n}_{D}} (\C(z) \e u) \|_{q} \leq C \| u \|_{W^{1, q}}\,.
 \end{displaymath}
 The last inequality implies~\eqref{e:integrability1}.
 
 In order to prove~\eqref{e:integrability2}, we first rewrite the Euler-Lagrange equation~\eqref{e:EL2} satisfied by~$u_{2}$. Namely, for every~$\psi \in W^{1, \tilde{p}'}(\Om; \R^{n})$ with~$\psi = 0 $ on~$\Gamma_{D}$ we have, after a simple algebraic manipulation,
 \begin{align}\label{e:24}
 \int_{\Om} & B_{z_{1}, \gamma} (u_2) {\, \cdot \, } \e \psi \, \di x 
 \\
 &
 = \int_{\Om} \C(z_{1}) \big(  F_{z_{1}, \gamma}^{-1} ( \Pi_{\M^{n}_{D}} (\C(z_{2}) \e u_{2}) ) - F^{-1}_{z_{1}, \gamma} ( \Pi_{\M^{n}_{D}} (\C(z_{1}) \e u_{2}) ) \big) \cdot \e \psi \, \di x \nonumber
 \\
 &
 \qquad + \int_{\Om} \C(z_{1}) \big( F_{z_{2}, \gamma}^{-1} ( \Pi_{\M^{n}_{D}} (\C(z_{2}) \e u_{2})) - F^{-1}_{z_{1}, \gamma} (\Pi_{\M^{n}_{D}} (\C (z_{2}) \e u_{2} ) ) \big) \cdot \e \psi \, \di x \nonumber
 \\
 &
 \qquad + \int_{\Om} \big( \C (z_{1}) - \C(z_{2}) \big) \big( \e u_{2} - F^{-1}_{z_{2}, \gamma} (\Pi_{\M^{n}_{D}} ( \C (z_{2}) \e u_{2})) \big) \cdot \e \psi \, \di x \nonumber
 \\
 &
\qquad  + \int_{\Om} \ell(z_{2}) f {\, \cdot \, } \psi \, \di x  + \int_{\Gamma_{N}} g {\, \cdot\,} \psi \, \di \HH^{n-1} \,. \nonumber
 \end{align}

Comparing~\eqref{e:24} with~\eqref{e:EL2} written for~$(z_{1}, u_{1})$, we deduce that~$u_{1}$ and~$u_{2}$ solve the same kind of equation, with a different right-hand side, always belonging to $W^{-1,\tilde{p}}(\Om; \R^{n})$. Thus, applying once more~\cite[Theorem~1.1]{Herzog}, we infer that there exists~$C>0$ independent of~$z_{1}, z_{2}$ and of~$\gamma$ such that for every $q \in   (2, \tilde{p}]\EEE$
\begin{align}\label{e:25}
\| u_{1} - u_{2} \|_{W^{1, q}} & \leq C \Big( \left\| \C(z_{1}) \big( F_{z_{1}, \gamma}^{-1} ( \Pi_{\M^{n}_{D}} (\C(z_{2}) \e u_{2}) ) - F^{-1}_{z_{1}, \gamma} ( \Pi_{\M^{n}_{D}} (\C(z_{1}) \e u_{2}) ) \big) \right \|_{W^{-1, q}} 
\\
&
\qquad+ \left \| \C(z_{1}) \big( F_{z_{2}, \gamma}^{-1} ( \Pi_{\M^{n}_{D}} (\C(z_{2}) \e u_{2})) - F^{-1}_{z_{1}, \gamma} (\Pi_{\M^{n}_{D}} (\C (z_{2}) \e u_{2} ) ) \big) \right \|_{W^{-1, q}} \nonumber
\\
&
\qquad + \left \| \big( \C (z_{1}) - \C(z_{2}) \big) \big( \e u_{2} - F^{-1}_{z_{2}, \gamma} (\Pi_{\M^{n}_{D}} ( \C (z_{2}) \e u_{2})) \big) \right \|_{W^{- 1, q}} \nonumber
\\
&
\qquad + \left \| \big(\ell(z_{1}) - \ell(z_{2}) \big) f \right \|_{W^{-1, q}}\Big) \nonumber
\\
&
=: C ( I_{1} + I_{2} + I_{3} + I_{4}) \,. \nonumber
\end{align}
By the Lipschitz continuity~\eqref{e:15} of $F^{-1}_{ z ,\gamma}$   and by~\eqref{e:integrability1} \EEE we deduce that
\begin{align}\label{e:I1}
I_{1} & \leq C \| (\C (z_{1}) - \C(z_{2})) \e u_{2} \|_{q} \leq C   \| z_{1} - z_{2} \|_{\infty} \EEE\| u \|_{W^{1, \tilde{p}}} 
\\
&
\leq C ( \| f \|_{p} + \| g \|_{p} + \| w \|_{W^{1, p}})   \| z_{1} - z_{2} \|_{\infty} \EEE  \,. \nonumber
\end{align}

Rewriting~\eqref{e:9} for $\q_{i} = F^{-1}_{z_{i} , \gamma} ( \Pi_{\M^{n}_{D}} (\C (z_{2}) \e u_{2}) ) $ we get that for a.e.~$x \in \Om$
\begin{align}\label{e:smth}
C_{1} | F^{-1}_{z_{1}, \gamma}  ( \Pi_{\M^{n}_{D}} & (\C (z_{2}) \e u_{2}) ) - F^{-1}_{z_{2}, \gamma} (\Pi_{\M^{n}_{D}} (\C (z_{2}) \e u_{2}) ) | 
\\
&
\leq  C_{2} | F^{-1}_{z_{2}, \gamma} (\Pi_{\M^{n}_{D}} (\C (z_{2}) \e u_{2}) ) | \, |z_{1} - z_{2} | + C_{3}  | z_{1} - z_{2} |\,. \nonumber
\end{align}
The  identification $\p_{2} = S_{\gamma, 3}(z_{2}) = F^{-1}_{z_{2}, \gamma} ( \Pi_{\M^{n}_{D}} (\C (z_{2} ) \e u_{2}) )$   and inequalities~\eqref{e:integrability1} and~\eqref{e:smth} \EEE imply that
\begin{equation}\label{e:I2}
I_{2} \leq C ( \| \p_{2} \|_{\tilde{p}} + 1)  \| z_{1} - z_{2} \|_{\infty} \EEE \leq C ( \| f \|_{p} + \| g \|_{p} + \| w \|_{W^{1, p}} + 1 )   \| z_{1} - z_{2} \|_{\infty} \EEE \,.
\end{equation}

As for~$I_{3}$, we simply use the Lipschitz continuity of~$\C(\cdot)$ and inequality~\eqref{e:integrability1} to show that
\begin{equation}\label{e:I3}
I_{3} \leq C   \| z_{1} - z_{2} \|_{\infty} \EEE( \| u_{2} \|_{W^{1, \tilde{p}}} + \| \p_{2}  \|_{\tilde{p}} ) \leq C ( \| f \|_{p} + \| g \|_{p} + \| w \|_{W^{1, p}})   \| z_{1} - z_{2} \|_{\infty} \EEE  \,.
\end{equation}
 In a similar way, since $\ell$ is Lipschitz continuous we obtain
\begin{equation}\label{e:I4}
I_{4} \leq C   \| z_{1} - z_{2} \|_{\infty} \EEE\| f \|_{p}\,. 
\end{equation}

Finally, inserting~\eqref{e:I1}--\eqref{e:I4} into~\eqref{e:25} we infer
\begin{equation}\label{e:uLip}
 \| u_{1} - u_{2} \|_{W^{1, q}} \leq C ( \| f \|_{p} + \| g \|_{p}  + \| w \|_{W^{1, p}} + 1 )   \| z_{1} - z_{2} \|_{\infty} \EEE\,.
\end{equation}
In order to conclude for~\eqref{e:integrability2}, we notice that   inequality~\eqref{e:9} \EEE tested with
\begin{displaymath}
\q_{i} = F^{-1}_{z_{i}(x), \gamma} \big( \Pi_{\M^{n}_{D}} ( \C(  z_{i}(x) \EEE ) \e u_{i}(x)) \big) = \p_{i}(x) \qquad \text{for a.e.~$x \in \Om$}
\end{displaymath}
and integrated over~$\Om$ implies
\begin{align*}
 \| \p_{1} - \p_{2} \|_{q} & \leq   C \big(  \| u_{1} \|_{W^{1,q}} \| z_{1} - z_{2}\|_{\infty} + \| u_{1} - u_{2} \|_{W^{1, q}} + \| \p_{2} \|_{q} \| z_{1} - z_{2} \|_{\infty} \big)
 \\
 &
   \leq C ( \| f \|_{p} + \| g \|_{p}  + \| w \|_{W^{1, p}} + 1)   \| z_{1} - z_{2} \|_{\infty} \,, \EEE
\end{align*}  
  where in the last inequality we have used~\eqref{e:uLip}. \EEE By the triangle inequality, we also estimate~$ \| \strain_{1} - \strain_{2} \|_{q}$, and the proof of~\eqref{e:integrability2} is complete.
\end{proof}

We are now ready to prove Theorem~\ref{t:regoptim}.

\begin{proof}[Proof of Theorem~\ref{t:regoptim}]
Let us fix~$\gamma \in(0,+\infty)$ and~$z, \varphi \in L^{\infty}(\Om)$. For $t \in \R$, let~$z_{t} \coloneq z + t\varphi$,~$(u_{t}, \strain_{t}, \p_{t})\coloneq S_{\gamma}( z_{t})$. The solution for~$t=0$ will be simply denoted with~$(u, \strain, \p)$. Moreover, we denote  with~$(v^{\varphi}_{\gamma}, \eeta^{\varphi}_{\gamma}, \qq^{\varphi}_{\gamma})$  the solution of~\eqref{e:der} and we set
\begin{displaymath}
\overline{v}_{t} \coloneq u_{t} - u - t  v^{\varphi}_{\gamma}  \,, \qquad \overline{\eeta}_{t} \coloneq \strain_{t} - \strain - t  \eeta^{\varphi}_{\gamma}  \,, \qquad \overline{\qq}_{t} \coloneq \p_{t} - \p - t  \qq^{\varphi}_{\gamma}  \,.
\end{displaymath} 
In what follows, we show that
\begin{equation}\label{e:26}
\| ( \overline{v}_{t}, \overline{\eeta}_{t}, \overline{\qq}_{t}) \|_{H^{1} \times L^{2} \times L^{2}} = o ( t ) \,,
\end{equation}
which implies the statement of the Theorem.

Writing the Euler-Lagrange equations satisfied by~$(u_{t}, \strain_{t}, \p_{t})$, $(u, \strain, \p )$, and  $(v^{\varphi}_{\gamma}, \eeta^{\varphi}_{\gamma}, \qq^{\varphi}_{\gamma})$  and subtracting the second and the third from the first one, we obtain, for every $(v, \eeta, \qq) \in \A(0)$,
\begin{displaymath}
\begin{split}
\int_{\Om} & \C(z_{t}) \strain_{t} {\, \cdot\, } \eeta \, \di x - \int_{\Om} \C (z) \strain {\, \cdot\,} \eeta \, \di x - t \int_{\Om} \C(z)  \eeta^{\varphi}_{\gamma}  {\, \cdot\,} \eeta \, \di x - t \int_{\Om} (\C'(z) \varphi) \strain {\, \cdot \, } \eeta \, \di x + \int_{\Om} \ha (z_{t}) \p_{t} {\, \cdot\,} \qq \, \di x
\\
&
 - \int_{\Om} \ha (z) \p {\, \cdot\,} \qq \, \di x - t \int_{\Om} \ha(z)  \qq^{\varphi}_{\gamma}  {\, \cdot \,} \qq \, \di x  - t  \int_{\Om} (\ha'(z) \varphi) \p {\, \cdot\,} \qq \, \di x + \int_{\Om} d(z_{t} ) \nabla_{\q} h_{\gamma}( \p_{t}) {\,\cdot\,} \qq \, \di x 
 \\
 &
 - \int_{\Om} d (z ) \nabla_{\q} h_{\gamma} (\p) {\, \cdot\,} \qq \, \di x - t \int_{\Om} \varphi \, d'(z) \nabla_{\q} h_{\gamma} (\p) {\, \cdot\,} \qq \, \di x - t \int_{\Om} d(z) \big( \nabla^{2}_{\q} h_{\gamma} (\p)  \qq^{\varphi}_{\gamma}  \big){ \, \cdot \,} \qq \, \di x 
 \\
 &
 - \int_{\Om} \ell(z_{t}) f {\, \cdot\,} v \, \di x + \int_{\Om} \ell(z) f {\, \cdot\,} v \, \di x + t \int_{\Om} \ell'(z) \varphi \, f{\, \cdot\,} v \, \di x = 0 \,. 
\end{split}
\end{displaymath}
By a simple algebraic manipulation, we rewrite the previous equality as
\begin{align}\label{e:27}
&0  =  \bigg(\int_{\Om}  \C(z_{t}) \strain_{t} {\, \cdot\, } \eeta \, \di x - \int_{\Om} \C (z) \strain {\, \cdot\,} \eeta \, \di x - t \int_{\Om} \C(z)  \eeta^{\varphi}_{\gamma}  {\, \cdot\,} \eeta \, \di x - t \int_{\Om} (\C'(z) \varphi) \strain {\, \cdot \, } \eeta \, \di x \bigg)
\\
&
\quad + \bigg( \int_{\Om} \ha (z_{t}) \p_{t} {\, \cdot\,} \qq \, \di x  - \int_{\Om} \ha (z) \p {\, \cdot\,} \qq \, \di x -  t  \int_{\Om} \ha(z)  \qq^{\varphi}_{\gamma}  {\, \cdot \,} \qq \, \di x  -  t \int_{\Om} (\ha'(z) \varphi) \p {\, \cdot\,} \qq \, \di x  \bigg) \nonumber
\\
&
\quad + \bigg( \int_{\Om}  \big( d(z_{t}) - d(z) - t \varphi \, d' (z)\big) \nabla_{\q} h_{\gamma} (\p) {\, \cdot\,} \qq \, \di x  \nonumber 
\\
&
\quad + \int_{\Om} \big( d(z_{t}) - d( z) \big) \big( \nabla_{\q} h_{\gamma} (\p_{t}) - \nabla_{\q} h_{\gamma} ( \p) \big) \cdot \qq \, \di x  \bigg) \nonumber
\\ 
&
\quad + \bigg( \int_{\Om} d(z) \big(\nabla_{\q} h_{\gamma} (\p_{t}) - \nabla_{\q} h_{\gamma} (\p) - t \nabla_{\q}^{2} h_{\gamma} (\p)  \qq^{\varphi}_{\gamma}  \big) {\, \cdot \,} \qq \, \di x \bigg) \nonumber
 \\
&
\quad  - \bigg(  \int_{\Om} \big( \ell(z_{t}) - \ell(z) - t \ell'(z) \varphi \big)  f{\, \cdot\,} v \, \di x \bigg)\nonumber
\\
&
= : I_{t, 1} + I_{t, 2} + I_{t, 3} + I_{t, 4}  + I_{t, 5} \, . \nonumber
\end{align}

Let us now estimate $I_{t, j}$, $j = 1, \ldots,  5 $. We write~$I_{t, 1}$ as
\begin{displaymath}
I_{t, 1} =  \int_{\Om}  \! \C(z) \overline{\eeta}_{t} {\, \cdot\, } \eeta \, \di x +\! \int_{\Om} \! (\C(z_{t}) - \C(z)) ( \strain_{t} - \strain) {\, \cdot \, } \eeta \, \di x
+ \! \int_{\Om} \! \big(\C(z_{t}) - \C(z) - t (\C'(z) \varphi)\big) \strain {\, \cdot\,} \eeta \, \di x .
\end{displaymath}
In a similar way, we have that
\begin{displaymath}
I_{t, 2} =  \int_{\Om}  \! \ha(z) \overline{\qq}_{t} {\, \cdot\, } \qq \, \di x + \! \int_{\Om} \! (\ha(z_{t}) - \ha(z)) ( \p_{t} - \p) {\, \cdot \, } \qq \, \di x
+ \! \int_{\Om} \! \big(\ha(z_{t}) - \ha(z) - t (\ha'(z) \varphi)\big) \p {\, \cdot\,} \qq \, \di x .
\end{displaymath}

As for~$I_{t, 4}$, since~$h_{\gamma} \in C^{\infty}( \M^{n}_{D})$, for every $t>0$ there exists $\boldsymbol{\xi}_{t}$ laying on the segment~$[\p, \p_{t}]$ such that
\begin{displaymath}
\begin{split}
I_{t, 4} & = \int_{\Om} d(z) \big(\nabla^{2}_{\q} h_{\gamma} (\boldsymbol{\xi}_{t}) ( \p_{t} - \p) - t \nabla_{\q}^{2} h_{\gamma} (\p)  \qq^{\varphi}_{\gamma}  \big) {\, \cdot \,} \qq \, \di x
\\
&
= \int_{\Om} d( z ) \big( \nabla^{2}_{\q} h_{\gamma} (\boldsymbol\xi_{t}) - \nabla^{2}_{\q} h_{\gamma} (\p)\big) ( \p_{t} - \p) {\, \cdot\,} \qq \, \di x + \int_{\Om} d (z) \nabla_{\q}^{2} h_{\gamma} (\p)  \overline{\qq}_{t} {\, \cdot \, } \qq \, \di x \,. 
\end{split}
\end{displaymath}

Inserting the previous equalities in~\eqref{e:27}, choosing the test function $(v, \eeta, \qq) = ( \overline{v}_{t}, \overline{\eeta}_{t}, \overline{\qq}_{t}) \in \A(0)$, using~\eqref{e:1}--\eqref{e:2}, the Lipschitz continuity of~$\C(\cdot)$,~$\ha( \cdot)$,~$d(\cdot)$, and~$\nabla_{\q} h_{\gamma}$, the convexity of~$h_{\gamma}$, and Lemma~\ref{l:integrability}, we obtain the estimate
\begin{align}  \label{e:28}
  \| ( \overline{v}_{t}, \overline{\eeta}_{t}, \overline{\qq}_{t} ) \|_{H^{1} \times L^{2}\times L^{2}}^{2} \EEE \leq & \ C_{\gamma} t^{2} \| \varphi \|_{\infty}^{2}  \| ( \overline{v}_{t}, \overline{\eeta}_{t}, \overline{\qq}_{t} ) \|_{H^{1} \times L^{2}\times L^{2}} \EEE  
\\
&
+  \int_{\Om} \big(\C(z_{t}) - \C(z) - t (\C'(z) \varphi)\big) \strain {\, \cdot\,} \overline{\eeta}_{t} \, \di x \nonumber
\\
&
+  \int_{\Om} \big(\ha(z_{t}) - \ha(z) - t (\ha'(z) \varphi)\big) \p {\, \cdot\,} \overline{\qq}_{t} \, \di x \nonumber
\\
&
+ \int_{\Om}  \big( d(z_{t}) - d(z) - t \varphi \, d' (z)\big) \nabla_{\q} h_{\gamma} (\p) {\, \cdot\,} \overline{\qq}_{t} \, \di x \nonumber
\\
&
+  \int_{\Om}d ( z) \big( \nabla^{2}_{\q} h_{\gamma} (\boldsymbol\xi_{t}) - \nabla^{2}_{\q} h_{\gamma} (\p)\big) ( \p_{t} - \p) {\, \cdot\,} \overline{\qq}_{t} \, \di x  \nonumber
\\
&
 -   \int_{\Om} \big( \ell(z_{t}) - \ell(z) - t \ell'(z) \varphi \big)  f{\, \cdot\,} \overline{v}_{t} \, \di x \,,\nonumber
\end{align}
for some positive constant~$C_{\gamma}$ dependent on~$\gamma \in (0,+\infty)$. In view of the regularity of~$\C( \cdot)$,~$\ha( \cdot)$,~$d(\cdot)$, and~$\ell(\cdot)$, we can continue in~\eqref{e:28} with
\begin{align}\label{e:29}
&   \| ( \overline{v}_{t},  \overline{\eeta}_{t}, \overline{\qq}_{t} ) \|_{H^{1} \times L^{2}\times L^{2}}^{2} \EEE  
\\
& 
\qquad \leq  \widetilde{C}_{\gamma} t^{2} \| \varphi\|_{\infty}^{2}\big(   \| ( u, \strain , \p) \|_{H^{1} \times L^{2} \times L^{2}}\EEE + 1 \big)  \| ( \overline{v}_{t}, \overline{\eeta}_{t}, \overline{\qq}_{t} ) \|_{H^{1} \times L^{2}\times L^{2}} \EEE  \nonumber
\\
&
\qquad\qquad  +  \int_{\Om} d( z ) \big( \nabla^{2}_{\q} h_{\gamma} (\boldsymbol\xi_{t}) - \nabla^{2}_{\q} h_{\gamma} (\p)\big) ( \p_{t} - \p) {\, \cdot\,} \overline{\qq}_{t} \, \di x  \nonumber
\\
&
\qquad \leq \widetilde{C}_{\gamma} t \| \varphi\|_{\infty}   \| ( \overline{v}_{t}, \overline{\eeta}_{t}, \overline{\qq}_{t} ) \|_{H^{1} \times L^{2}\times L^{2}}^{2} \EEE \Big( t \| \varphi \|_{\infty} \big(   \| ( u, \strain , \p) \|_{H^{1} \times L^{2} \times L^{2}} \EEE + 1 \big) \nonumber
\\
&
\qquad\qquad + \|  \nabla^{2}_{\q} h_{\gamma} (\boldsymbol\xi_{t}) - \nabla^{2}_{\q} h_{\gamma} (\p) \|_{\nu} \Big)  \nonumber
\end{align}
for some~$\widetilde{C}_{\gamma} > 0 $ and some $\nu \in (1, +\infty)$. In order to conclude for~\eqref{e:26} we are led to show that the last term on the right-hand side of~\eqref{e:29} tends to~$0$ as $t\to 0$. To do this, we explicitly write~$\nabla^{2}_{\q} h_{\gamma}(\q)$ for $\q \in \M^{n}_{D}$:
\begin{equation}\label{e:hessian}
\nabla^{2}_{\q} h_{\gamma} (\q) = \frac{1}{\sqrt{|\q|^{2} + \frac{1}{\gamma^{2}}}} \left( \id - \frac{\q \otimes \q}{ |\q|^{2} + \frac{1}{\gamma^{2}}} \right).
\end{equation}
Since $ \p_{t} \to \p $ in $L^{2} (\Om; \M^{n}_{D})$ as $t\to0$, up to a subsequence we can assume that $\p_{t} \to \p$ a.e.~in~$\Om$. Hence,~$\boldsymbol\xi_{t} \to \p$ and ~$\nabla^{2}_{\q} h_{\gamma} (\boldsymbol\xi_{t}) \to \nabla^{2}_{\q} h_{\gamma}(\p)$ a.e.~in~$\Om$. In view of formula~\eqref{e:hessian} of~$\nabla^{2}_{\q} h_{\gamma}$, we have that $|  \nabla^{2}_{\q}  h_{\gamma} (\boldsymbol\xi_{t}) | \leq 2\gamma$ in~$\Om$. Thus, the dominated convergence theorem implies that $\|  \nabla^{2}_{\q} h_{\gamma} (\boldsymbol\xi_{t}) - \nabla^{2}_{\q} h_{\gamma} (\p) \|_{\nu} \to 0$ as $t\to 0$. This, together with~\eqref{e:29}, concludes the proof of~\eqref{e:26}.
\end{proof}

We conclude this section with the proof of Corollary~\ref{c:regoptim}.

\begin{proof}[Proof of Corollary~\ref{c:regoptim}]
 Let us define the triple $(\overline{u}_{\gamma}, \overline{\strain}_{\gamma}, \overline{\p}_{\gamma}) \in \A(0)$ as the unique solution of
\begin{align}\label{e:adjoint}
\min\, \bigg\{ \frac12 \int_{\Om} \C(z_{\gamma}) \strain {\, \cdot\, } \strain \, \di x & + \frac12 \int_{\Om} \ha (z_{\gamma}) \p{\, \cdot\,} \p \, \di x 
+ \frac12 \int_{\Om} d(z_{\gamma}) \big(\nabla^{2}_{\q} h_{\gamma} (p_{\gamma}) \p \big) {\, \cdot\,} \p \, \di x 
\\
&
- \int_{\Om} \ell (z_{\gamma}) f {\, \cdot\,} u \, \di x - \int_{\Gamma_{N}} g {\, \cdot\,} u \, \di \HH^{n-1}: \, (u, \strain, \p) \in \A(0) \bigg\}\,. \nonumber
\end{align}
In particular, the triple $(\overline{u}_{\gamma}, \overline{\strain}_{\gamma}, \overline{\p}_{\gamma})$ satisfies~\eqref{e:optreg2}. In what follows, we show that $(\overline{u}_{\gamma}, \overline{\strain}_{\gamma}, \overline{\p}_{\gamma})$ also satisfies~\eqref{e:optreg}.

Given $\varphi \in H^{1} (\Om) \cap L^{\infty}(\Om)$,  we set  $( v^{\varphi}_{\gamma}, \eeta^{\varphi}_{\gamma}, \qq^{\varphi}_{\gamma})  \coloneq S'_{\gamma}(z_{\gamma}) [\varphi]$.  By Theorem~\ref{t:regoptim},  $( v^{\varphi}_{\gamma}, \eeta^{\varphi}_{\gamma}, \qq^{\varphi}_{\gamma})$  is a solution of~\eqref{e:der}, which in turn is equivalent to the Euler-Lagrange equation 
\begin{align}
 & \int_{\Om} \C(z_{\gamma})  \eeta^{\varphi}_{\gamma} {\, \cdot\, }\eeta \, \di x 
+ \int_{\Om} \ha(z_{\gamma}) \qq^{\varphi}_{\gamma} { \, \cdot \,} \qq \, \di x 
+ \int_{\Om} ( \C'(z_{\gamma}) \varphi) \strain_{\gamma} {\, \cdot\, } \eeta \, \di x  
+ \int_{\Om} (\ha'(z_{\gamma}) \varphi) \p_{\gamma} {\, \cdot\,} \qq \, \di x \label{e:optreg3}
\\
&\! \quad
+ \int_{\Om} \varphi \, d'(z_{\gamma}) \nabla_{\q} h_{\gamma}(\p_{\gamma})  {\, \cdot\,} \qq \, \di x
+ \int_{\Om} d (z_{\gamma} ) \big( \nabla^{2}_{\q} h_{\gamma} (\p_{\gamma}) \qq^{\varphi}_{\gamma} \big){ \, \cdot \,} \qq \, \di x  - \int_{\Om} \!\! \varphi\, \ell' (z) \, f{\, \cdot\,} v \, \di x  = 0 \,,\nonumber 
\end{align}
for every $(v, \eeta, \qq) \in \A(0)$. 

 For~$t \in \R$ we consider $z^{t}_{\gamma} \coloneq z_{\gamma} + t\varphi$ and denote with $(u_{\gamma}^{t}, \strain_{\gamma}^{t}, \p_{\gamma}^{t})  = S_{\gamma}(z_{\gamma}^{t})$ the corresponding solution of the forward problem~\eqref{e:optimization4}. By optimality of~$z_{\gamma}$ we have that $\J_{\delta} (z_{\gamma}, u_{\gamma} )\leq \J_{\delta}( z_{\gamma}^{t}, u_{\gamma}^{t})$. We divide the previous expression by~$t$ and pass to the limit as~$t\to 0$ taking into account the differentiability of the control-to-state operator~$S_{\gamma}$ from Theorem~\ref{t:regoptim}, so that we deduce
\begin{align}
& \int_{\Om}  \varphi \, \ell'(z_{\gamma} ) f {\, \cdot\,} u_{\gamma}\, \di x  + \int_{\Om} \ell( z_{\gamma}) f{\, \cdot\,} v^{\varphi}_{\gamma} \, \di x + \int_{\Gamma_{N}} g {\, \cdot\,} v^{\varphi}_{\gamma} \, \di \HH^{n-1} \label{e:optreg4}
 \\
 &\quad
  + \int_{\Om}  \delta \nabla{z}_{\gamma} {\, \cdot \, } \nabla{\varphi}    + \frac{1}{\delta}\,\varphi  (z_{\gamma} (1 - z_{\gamma})^{2} - z_{\gamma}^{2} (1 - z_{\gamma})) \, \di x  = 0 \,. \nonumber
\end{align}

Since $( v^{\varphi}_{\gamma}, \eeta^{\varphi}_{\gamma}, \qq^{\varphi}_{\gamma}) \in \A(0)$, from~\eqref{e:optreg2} we infer that
\begin{align}
\int_{\Om} \ell(z_{\gamma})  f{\, \cdot\,} v^{\varphi}_{\gamma}\, \di x + \int_{\Gamma_{N}} g{\, \cdot\,} v^{\varphi}_{\gamma}\, \di \HH^{n-1} = & \int_{\Om} \C(z_{\gamma}) \overline{\strain}_{\gamma} {\, \cdot\,} \eeta^{\varphi}_{\gamma} \, \di x + \int_{\Om} \ha(z_{\gamma} ) \overline{\p}_{\gamma}{\, \cdot\,} \qq^{\varphi}_{\gamma}\, \di x  \label{e:someeq2}
\\
&
+ \int_{\Om} d(z_{\gamma}) \big( \nabla_{\q}^{2} h_{\gamma}(\p_{\gamma}) \overline{\p}_{\gamma}\big) {\, \cdot\,} \qq^{\varphi}_{\gamma} \, \di x \nonumber\,.
\end{align}
As also $(\overline{u}_{\gamma}, \overline{\strain}_{\gamma}, \overline{\p}_{\gamma})$ belongs to~$\A(0)$, by~\eqref{e:optreg3} we continue in~\eqref{e:someeq2} with
\begin{align}
\int_{\Om} \! \ell(z_{\gamma})  f{\, \cdot\,} v^{\varphi}_{\gamma}\, \di x + \int_{\Gamma_{N}} \!\!\! g{\, \cdot\,} v^{\varphi}_{\gamma}\, \di \HH^{n-1} = & -\! \int_{\Om} ( \C'(z_{\gamma}) \varphi) \strain_{\gamma} {\, \cdot\, } \overline{\strain}_{\gamma} \, \di x  
- \! \int_{\Om} (\ha'(z_{\gamma}) \varphi) \p_{\gamma} {\, \cdot\,} \overline{\p}_{\gamma} \, \di x \label{e:someeq3}
\\
& 
-\! \int_{\Om} \varphi \, d'(z_{\gamma}) \nabla_{\q} h_{\gamma}(\p_{\gamma})  {\, \cdot\,} \overline{\p}_{\gamma} \, \di x
+\! \int_{\Om} \! \varphi \ell' (z) \, f{\, \cdot\,} \overline{u}_{\gamma} \, \di x \,. \nonumber
\end{align}
Combining~\eqref{e:optreg4} and \eqref{e:someeq3} we obtain~\eqref{e:optreg}, and the proof is concluded.
\end{proof}


\section{Optimality conditions for $\gamma \to + \infty$}
\label{s:optimality}

This section is devoted to the computation of the optimality conditions for~\eqref{e:optimization1}--\eqref{e:optimization2}. Such conditions are obtained by passing to the limit in~\eqref{e:optreg2} and~\eqref{e:optreg}. As the control-to-state operator for~\eqref{e:optimization1}--\eqref{e:optimization2} is not differentiable anymore and quantities such as~$\nabla_{\q} h_{\gamma}$ and~$\nabla^{2}_{\q} h_{\gamma}$ appearing in formulas~\eqref{e:F}--\eqref{e:der} degenerate as~$\gamma \to +\infty$, the limit passage is not completely trivial and deserves some further analysis. This is the content of the following theorem.


\begin{theorem}[First-order conditions]
\label{t:lim_optim}
Let $p \in (2, +\infty)$, $f\in L^{p}(\Om; \R^{n})$, $g \in L^{p}(\Gamma_{N}; \R^{n})$, and $w \in W^{1, p}( \Om; \R^{n})$. For every $\gamma \in (0, +\infty)$ let $z_{\gamma} \in H^{1}(\Om; [0,1])$ be a solution of~\eqref{e:optimization3}--\eqref{e:optimization4}, with corresponding state variable~$(u_{\gamma}, \strain_{\gamma}, \p_{\gamma}) \in \A(w)$. Then, there exists $z \in H^{1}(\Om; [0,1])$ solution of~\eqref{e:optimization1}--\eqref{e:optimization2} with corresponding state variable~$(u, \strain, \p) \in \A(w)$ such that, up to a subsequence, $z_{\gamma} \rightharpoonup z$ weakly in~$H^{1}(\Om)$ and $(u_{\gamma}, \strain_{\gamma}, \p_{\gamma}) \to (u, \strain, \p)$ in $H^{1}(\Om; \R^{n}) \times L^{2} (\Om; \M^{n}_{S}) \times L^{2}(\Om; \M^{n}_{D})$ as $\gamma \to +\infty$. 

Moreover, there exists~$\boldsymbol{\rho} \in L^{\infty}(\Om)$ such that for every $( v, \eeta, \qq) \in \A(0)$ 
\begin{eqnarray}
& \displaystyle \,\,\, \int_{\Om} \C (z) \strain {\, \cdot\,} \eeta \, \di x + \int_{\Om} \ha(z) \p {\, \cdot\,} \qq \, \di x + \int_{\Om} \boldsymbol\rho {\, \cdot\,} \qq \, \di x - \int_{\Om} \ell(z)  f {\, \cdot \, } v \, \di x - \int_{\Gamma_{N}} g {\, \cdot\, } v \, \di \HH^{n-1} = 0\,. 
 \label{e:30} \\[1mm]
 & \displaystyle \vphantom{\int} \boldsymbol{\rho} {\, \cdot\, } \p = d( z ) | \p | \quad \text{in $\Om$} \,, \qquad \p=0 \quad \text{in~$\{ |\boldsymbol{\rho}| < d( z ) \}$}\,. \label{e:31}
 \end{eqnarray}

Finally, there exist the adjoint variables $(\overline{u}, \overline{\strain}, \overline{\p}) \in \A(0)$ and~$\boldsymbol\pi \in L^{2}(\Om; \M^{n}_{D})$ such that for every $(v, \eeta, \qq) \in \A(0)$ and every $\varphi \in L^{\infty}(\Om) \cap H^{1}(\Om)$ the following hold:
\begin{align}
 & \int_{\Om} \C(z)  \overline{\strain} {\, \cdot\, }\eeta \, \di x 
+ \int_{\Om} \ha(z) \overline{\p} { \, \cdot \,} \qq \, \di x 
+ \int_{\Om} \boldsymbol \pi {\, \cdot\,} \qq \, \di x  - \int_{\Om} \ell (z) \, f{\, \cdot\,} v \, \di x  - \int_{\Gamma_{N}} g {\, \cdot\,} v \, \di \HH^{n-1} = 0 \,, \label{e:32}
\\[1mm]
&  \int_{\Om}  \varphi \, \ell'(z ) f {\, \cdot\,}( \overline{ u} + u ) \, \di x  -\int_{\Om} \big( \C' (z) \varphi\big) \overline{\strain}{\, \cdot \, } \strain  \, \di x   \label{e:32.1}
 \\
 & \qquad\qquad
 - \int_{\Om} \big( \ha'(z) \varphi \big) \overline{\p} {\, \cdot \,} \p \, \di x 
   - \int_{\Om} \varphi \, \frac{d' ( z)}{d(z)} \boldsymbol \rho {\, \cdot\,} \overline{\p}_{\gamma} \, \di x  \nonumber
    \\
 &
 \qquad\qquad
 + \int_{\Om}  \delta \nabla{z} {\, \cdot \, } \nabla{\varphi}  + \frac{1}{\delta}\,\varphi  (z (1 - z )^{2} - z^{2} (1 - z )) \, \di x  = 0 \,. \nonumber\\
&
 \vphantom{\int} \boldsymbol \pi {\, \cdot\,} \p = 0 \quad\text{a.e.~in~$\Om$} \,, \qquad \boldsymbol \pi {\, \cdot\,} \overline{\p} \geq 0 \quad \text{a.e.~in~$\Om$} \,, \qquad \overline{\p} = 0 \quad \text{a.e.~in~$\{ |\boldsymbol\rho| < d(z) \}$}\,. \label{e:33}
\end{align}
\end{theorem}

\begin{proof}
The equalities~\eqref{e:30}--\eqref{e:31} are equivalent to the minimality of the triple~$(u, \strain, \p)$. For later use, we remark that equation~\eqref{e:30} implies that
\begin{displaymath}
\boldsymbol\rho = \Pi_{\M^{n}_{D}} \big( \C(z) \strain - \ha(z) \p \big)\,.
\end{displaymath}
In a similar way, the minimality of~$( u_{\gamma}, \strain_{\gamma}, \p_{\gamma})$ implies that
\begin{displaymath}
d( z_{\gamma} ) \nabla_{\q} h_{\gamma} ( \p_{\gamma}) = \Pi_{\M^{n}_{D}} \big( \C (z_{\gamma}) \strain_{\gamma} - \ha(z_{\gamma}) \p_{\gamma} \big) \,.
\end{displaymath}
As $(u_{\gamma}, \strain_{\gamma}, \p_{\gamma}) \to (u, \strain, \p)$ in $H^{1}(\Om;\R^{n}) \times L^{2}(\Om; \M^{n}_{S}) \times L^{2}(\Om; \M^{n}_{D})$ and $z_{\gamma} \rightharpoonup z$ weakly in $H^{1}(\Om; [0,1])$, we deduce from the previous equalities that $d(z_{\gamma})\nabla_{\q} h_{\gamma} ( \p_{\gamma}) \to \boldsymbol\rho$ in $L^{2}(\Om; \M^{n}_{D})$.

Let us prove~\eqref{e:32}--\eqref{e:32.1}. For $\gamma \in (0, +\infty)$, in view of Corollary~\ref{c:regoptim} we know that there exists~$(\overline{u}_{\gamma} , \overline{\strain}_{\gamma}, \overline{\p}_{\gamma}) \in \A(0)$ such that~\eqref{e:optreg2}--\eqref{e:optreg} hold.
From~\eqref{e:optreg2} we immediately deduce that $(\overline{u}_{\gamma}, \overline{\strain}_{\gamma}, \overline{\p}_{\gamma})$ is bounded in~$H^{1}(\Om; \R^{n}) \times L^{2}(\Om; \M^{n}_{S}) \times L^{2}(\Om; \M^{n}_{D})$. Hence, up to a subsequence,~$(\overline{u}_{\gamma}, \overline{\strain}_{\gamma}, \overline{\p}_{\gamma})$ converges weakly to some~$(\overline{u}, \overline{\strain}, \overline{\p}) \in \A(0)$. In particular, passing to the limit in~\eqref{e:optreg} as~$\gamma \to +\infty$ we infer~\eqref{e:32.1}.

By testing~\eqref{e:optreg2} with~$(0, \eeta, -\eeta)$ for~$\eeta \in L^{2}(\Om; \M^{n}_{D})$, we get that 
\begin{align}\label{e:35}
\boldsymbol\pi_{\gamma}  & \coloneq d(z_{\gamma}) \nabla^{2}_{\q} h_{\gamma} (\p_{\gamma}) \overline{\p}_{\gamma}  = \Pi_{\M^{n}_{D}} \big( \C ( z_{\gamma} ) \overline\strain_{\gamma} - \ha( z_{\gamma}) \overline{\p}_{\gamma} \big) \,. 
\end{align} 
In view of the convergences shown above, from~\eqref{e:35} we infer that $\boldsymbol\pi_{\gamma} $ converges weakly in $L^{2}(\Om; \M^{n}_{D})$ to
\begin{equation}\label{e:pi}
\boldsymbol\pi \coloneq \Pi_{\M^{n}_{D}} \big( \C ( z ) \overline{\strain} - \ha( z) \overline{\p} \big) \,.
\end{equation}
Therefore, passing to the limit in~\eqref{e:optreg2} we obtain~\eqref{e:32}. 

We now prove the second inequality in~\eqref{e:33}. For every $\gamma \in (0,+\infty)$ and every $\psi \in C^{\infty}_{c}(\Om)$ with $\psi \geq 0$, we test~\eqref{e:optreg2} with the triple $(\psi \overline{u}_{\gamma}, \psi \overline{\strain}_{\gamma} + \overline{u}_{\gamma} \odot \nabla\psi , \psi \overline{\p}_{\gamma}) \in \A(0)$, obtaining
 \begin{align}
0 = & \ \int_{\Om} \C(z_{\gamma}) \overline{\strain}_{\gamma} {\, \cdot\,} (\psi \overline{\strain}_{\gamma} + \overline{u}_{\gamma} \odot \nabla{\psi} ) \, \di x + \int_{\Om} \ha( z_{\gamma} ) \overline{\p}_{\gamma} {\, \cdot \,} ( \psi \overline{\p}_{\gamma}) \, \di x \label{e:36}
\\
&
   - \int_{\Om} \ell (z_{\gamma}) f {\, \cdot\,} ( \psi \overline{u}_{\gamma}) \, \di x - \int_{\Gamma_{N}} g {\, \cdot\,} (\psi \overline{u}_{\gamma}) \, \di \HH^{n-1} \EEE + \int_{\Om} \boldsymbol\pi_{\gamma} {\, \cdot \,} (\psi \overline{\p}_{\gamma} ) \, \di x \,. \nonumber 
\end{align}
Since $\psi \geq 0$,~$h_{\gamma}$ is convex,  and~\eqref{e:35} holds, we notice that the last term on the right-hand side of~\eqref{e:36} is positive. Recalling that~$\C$,~$\ha$, and~$d$ are of class~$C^{1}$, we can pass to the liminf in~\eqref{e:36} obtaining
 \begin{align}\label{e:37}
0 & \geq   \int_{\Om} \C(z) \overline{\strain}  {\, \cdot\,} (\psi \overline{\strain}  + \overline{u}  \odot \nabla{\psi} ) \di x + \! \int_{\Om} \ha( z ) \overline{\p}  {\, \cdot \,} ( \psi \overline{\p} )  \di x
\\
&
\qquad 
  - \int_{\Om}  \ell (z) {\, \cdot\,} ( \psi \overline{u}) \, \di x  - \int_{\Gamma_{N}} g {\, \cdot\,} (\psi \overline{u}) \, \di \HH^{n-1} \EEE
= - \int_{\Om} \boldsymbol\pi {\, \cdot\,} (\psi \overline{\p} ) \, \di x \,, \nonumber
\end{align}
where, in the last equality, we have used~\eqref{e:32} with the test~$( \psi \overline{u} , \psi \overline{\strain} + \overline{u} \odot \nabla \psi, \psi \overline{\p}) \in \A(0)$. The arbitrariness of~$\psi \geq 0$ in~\eqref{e:37} implies that $\boldsymbol\pi {\, \cdot \,} \overline \p \geq 0$ a.e.~in~$\Om$.

Let us show the first equality in~\eqref{e:33}. To do this, we set $J_{\gamma} \coloneq \{ | \p_{\gamma}| \leq \frac{1}{\sqrt{\gamma}} \} $ and estimate~$\| \boldsymbol\pi_{\gamma} {\, \cdot\,} \p_{\gamma} \|_{1}^{2}$ as follows:
\begin{align}
\| \boldsymbol \pi_{\gamma}  {\, \cdot \,} \p_{\gamma} \|_{1}^{2}  & = \Bigg ( \int_{\Om} \frac{ d(z_{\gamma} ) }{\sqrt{| \p_{\gamma}|^{2} + \frac{1}{\gamma^{2}}}}  \Bigg| ( \p_{\gamma} {\, \cdot \,} \overline{\p}_{\gamma} ) - \frac{ | \p_{\gamma} |^{2} (  \overline{\p}_{\gamma}  {\, \cdot\,} \p_{\gamma}) }{ | \p_{\gamma}|^{2} + \frac{1}{\gamma^{2}} } \Bigg| \, \di x \Bigg ) ^{2} \label{e:40}
\\
& 
\leq 2  \Bigg ( \int_{ J_{\gamma} } \frac{ d(z_{\gamma} ) }{\sqrt{| \p_{\gamma}|^{2} + \frac{1}{\gamma^{2}}}}  \Bigg| ( \p_{\gamma} {\, \cdot \,}  \overline{\p}_{\gamma}  )  - \frac{ |\p_{\gamma} |^{2} (  \overline{\p}_{\gamma}   {\, \cdot\,} \p_{\gamma} ) }{ | \p_{\gamma}|^{2} + \frac{1}{\gamma^{2}} } \Bigg| \, \di x \Bigg ) ^{2} \nonumber
\\
&
\qquad + 2 \Bigg ( \int_{  \Om \setminus J_{\gamma} } \frac{ d(z_{\gamma} ) }{\sqrt{| \p_{\gamma}|^{2} + \frac{1}{\gamma^{2}}}}  \Bigg| (  \p_{\gamma} {\, \cdot \,}  \overline{\p}_{\gamma}  ) - \frac{ | \p_{\gamma} |^{2} (  \overline{\p}_{\gamma}  {\, \cdot\,} \p_{\gamma} ) }{ | \p_{\gamma}|^{2} + \frac{1}{\gamma^{2}} } \Bigg| \, \di x \Bigg ) ^{2} 
= : I_{\gamma, 1} + I_{\gamma, 2} \,. \nonumber
\end{align}
We now show that~$I_{\gamma, 1}$ and~$I_{\gamma, 2}$ tend to~$0$ separately. For~$I_{\gamma, 1}$, by H\"older and Cauchy inequalities and by the continuity of~$d$ we have that there exists~$C>0$ independent of~$\gamma$ such that
\begin{align}\label{e:41}
I_{\gamma, 1} & \leq 2 |\Om| \int_{ J_{\gamma}} \frac{ d^{2}(z_{\gamma} )   |  \overline{\p}_{\gamma}  |^{2} | \p_{\gamma} |^{2} }{ | \p_{\gamma} |^{2} + \frac{1}{\gamma^{2}} } \, \Bigg | 1 - \frac{ | \p_{\gamma}|^{2}}{ | \p_{\gamma} | ^{2} + \frac{1}{\gamma^{2}}} \Bigg| ^{2} \di x  
\\
&
\leq C | \Om| \int_{ J_{\gamma}} \frac{  d(z_{\gamma} )  | \overline{\p}_{\gamma} |^{2} |\p_{\gamma}|}{ \sqrt{ |\p_{\gamma}|^{2} + \frac{1}{\gamma^{2}}}} \, \Bigg| 1 - \frac{| \p_{\gamma}|^{2}}{ | \p_{\gamma} |^{2} + \frac{1}{\gamma^{2}}}\Bigg|\, \di x \nonumber
\\
&
\leq  \frac{C| \Om|}{\sqrt{\gamma}} \int_{ J_{\gamma}} \frac{ d(z_{\gamma} ) }{\sqrt{ |\p_{\gamma}|^{2} + \frac{1}{\gamma^{2}}}} \, \Bigg| |  \overline{\p}_{\gamma}  |^{2} - \frac{ ( \p_{\gamma} {\, \cdot\,}  \overline{\p}_{\gamma}  )^{2} }{ | \p_{\gamma}|^{2} + \frac{1}{\gamma^{2}}}  \Bigg| \, \di x \,. \nonumber
\end{align}
In order to conclude that $I_{\gamma, 1} \to 0$, we write explicitly $| \boldsymbol\pi_{\gamma} {\, \cdot\,} \overline{\p}_{\gamma}  |$:
\begin{displaymath}
| \boldsymbol\pi_{\gamma} {\, \cdot\,} \overline{\p}_{\gamma} | = \frac{ d(z_{\gamma} ) }{ \sqrt{ | \p_{\gamma} |^{2} + \frac{1}{\gamma^{2}}}} \, \Bigg| | \overline{\p}_{\gamma} | ^{2} - \frac{ ( \p_{\gamma} {\, \cdot\,}  \overline{\p}_{\gamma} )^{2}}{ | \p_{\gamma} |^{2} + \frac{1}{\gamma^{2}}} \Bigg| \,.
\end{displaymath}
Therefore, we can continue in~\eqref{e:41} with
\begin{displaymath}
I_{\gamma, 1} \leq \frac{C | \Om| }{\sqrt{\gamma}} \int_{\Om} |  \boldsymbol\pi_{\gamma}  {\, \cdot\,}  \overline{\p}_{\gamma}  | \, \di x 
\end{displaymath}
which implies that $I_{\gamma, 1} \to 0$ as $\gamma \to + \infty$, as $ \boldsymbol\pi_{\gamma}$ and~$ \overline{\p}_{\gamma}$ are bounded in~$L^{2}(\Om; \M^{n}_{D})$.

As for~$I_{\gamma, 2}$, instead, by the Cauchy inequality we have
\begin{equation}\label{e:44}
\begin{split}
I_{\gamma, 2} & \leq C \left( \int_{\Om \setminus J_{\gamma}} | \overline{\p}_{\gamma} | \, \Bigg| 1  -  \frac{ | \p_{\gamma}|^{2}}{ | \p_{\gamma}|^{2} + \frac{1}{\gamma^{2}}} \Bigg| \, \di x \right)^{2}
\leq C  \left( \int_{\Om\setminus J_{\gamma}}  \frac{ | \overline{\p}_{\gamma} |}{ \gamma^{2} |\p_{\gamma}|^{2} + 1} \right)^{2} \leq \frac{C}{\gamma^{2}} \|  \overline{\p}_{\gamma}  \|^{2}_{1} \,.
\end{split}
\end{equation}
Thus, also~$I_{\gamma, 2}$ tends to~$0$ as~$\gamma \to +\infty$.

All in all, we deduce from~\eqref{e:40} that $ \boldsymbol\pi_{\gamma}  {\, \cdot\,} \p_{\gamma} \to 0$ in~$L^{1}(\Om)$. Furthermore, we know that $ \boldsymbol\pi_{\gamma} \rightharpoonup \boldsymbol\pi $ weakly in~$L^{2}(\Om; \M^{n}_{D})$ and $\p_{\gamma} \to \p$ in~$L^{2}(\Om; \M^{n}_{D})$, so that $ \boldsymbol\pi_{\gamma} {\, \cdot\,} \p_{\gamma} \rightharpoonup \boldsymbol\pi {\, \cdot\,} \p $ weakly in~$L^{1}(\Om)$, which implies $ \boldsymbol\pi  {\,\cdot\,} \p = 0$ in~$\Om$.

We now check with the third inequality in~\eqref{e:33}. Since $| \boldsymbol\rho| \leq d(z)$, the function
\begin{displaymath}
\qq \mapsto \int_{\Om} |\qq| \big( d( z ) - | \boldsymbol\rho| \big) \, \di x 
\end{displaymath}
is lower semicontinuous w.r.t.~the weak topology of~$L^{2}(\Om; \M^{n}_{D})$. Hence, 
\begin{displaymath}
\begin{split}
0 & \leq \int_{\Om} | \overline{\p} | \big( d(z) - | \boldsymbol\rho| \big) \, \di x \leq \liminf_{\gamma \to + \infty} \int_{\Om} | \overline{\p}_{\gamma}  | \big( d( z ) - | \boldsymbol\rho| \big) \, \di x 
\\
& \leq \limsup_{\gamma \to + \infty} \int_{\Om} |  \overline{\p}_{\gamma} | \, | d(z) - d(z_{\gamma}) | \, \di x + \limsup_{\gamma \to +\infty} \int_{\Om} |  \overline{\p}_{\gamma} | \, \big( d( z_{\gamma} ) - d(z_{\gamma})  | \nabla_{\q} h_{\gamma} (\p_{\gamma}) | \big) \, \di x 
\\
&
\qquad +  \limsup_{\gamma \to +\infty} \int_{\Om} |  \overline{\p}_{\gamma}  | ( d(z_{\gamma})| \nabla_{\q} h_{\gamma} (\p_{\gamma}) | - | \boldsymbol\rho| ) \, \di x \,.
\end{split}
\end{displaymath}
Since $ \overline{\p}_{\gamma}$ is bounded in~$L^{2}(\Om; \M^{n}_{D})$, $d(z_{\gamma})\nabla_{\q} h_{\gamma} (\p_{\gamma}) \to \boldsymbol\rho$ in~$L^{2}(\Om; \M^{2}_{D})$, and~$z_{\gamma} \rightharpoonup z$ weakly in~$H^{1}(\Om; [0,1])$, the previous inequality reduces to
\begin{align}\label{e:42}
0 & \leq \int_{\Om} |  \overline{\p}  | \big( d(z) - | \boldsymbol\rho| \big) \, \di x \leq \limsup_{\gamma \to +\infty} \int_{\Om} |  \overline{\p}_{\gamma} | \, \big( d( z_{\gamma} ) - d(z_{\gamma}) | \nabla_{\q} h_{\gamma} (\p_{\gamma}) | \big) \, \di x
\\
&
= \limsup_{\gamma \to +\infty} \int_{\Om} d(z_{\gamma} ) | \overline{\p}_{\gamma}  | \left( 1 - \frac{ | \p_{\gamma} | }{ \sqrt{| \p_{\gamma} |^{2} + \frac{1}{\gamma^{2}}}} \right)\, \di x  \nonumber
\\
&
\leq  \limsup_{\gamma \to +\infty} \int_{J_{\gamma}} d(z_{\gamma} )  |  \overline{\p}_{\gamma}  | \left( 1 - \frac{ | \p_{\gamma} | }{ \sqrt{| \p_{\gamma} |^{2} + \frac{1}{\gamma^{2}}}} \right)\, \di x  \nonumber
\\
&
\qquad +  \limsup_{\gamma \to +\infty} \int_{\Om \setminus J_{\gamma}} d(z_{\gamma} ) |  \overline{\p}_{\gamma}  | \left( 1 - \frac{ | \p_{\gamma} | }{ \sqrt{| \p_{\gamma} |^{2} + \frac{1}{\gamma^{2}}}} \right)\, \di x  
 \,, \nonumber
\end{align}
where in the second line we have used the explicit expression of~$\nabla_{\q} h_{\gamma} ( \p_{\gamma})$. Arguing as in~\eqref{e:44} we deduce that the second integral on the right-hand side of~\eqref{e:42} tends to~$0$ as $\gamma \to +\infty$. 

As for the integral on~$J_{\gamma}$, instead, by the Cauchy inequality and the continuity of~$d$ there exists a positive constant~$C$ independent of~$\gamma$ such that
\begin{align}\label{e:45}
\int_{J_{\gamma}} d(z_{\gamma} )  |  \overline{\p}_{\gamma}  | \left( 1 - \frac{ | \p_{\gamma} | }{ \sqrt{| \p_{\gamma} |^{2} + \frac{1}{\gamma^{2}}}} \right)\, \di x &  \leq \int_{J_{\gamma}} d(z_{\gamma} ) \left(  |  \overline{\p}_{\gamma}  | - \frac{ |  \p_{\gamma} {\, \cdot\,}  \overline{\p}_{\gamma}  | }{ \sqrt{| \p_{\gamma} |^{2} + \frac{1}{\gamma^{2}}}} \right)\, \di x  
\\
&
\leq C \int_{J_{\gamma}}  \left(  |  \overline{\p}_{\gamma} |^{2} - \frac{ |  \p_{\gamma} {\, \cdot\,}  \overline{\p}_{\gamma}  |^{2} }{  | \p_{\gamma} |^{2} + \frac{1}{\gamma^{2}}} \right)\, \di x  \,. \nonumber
\end{align} 
  On the set $J_{\gamma}$ we have that\EEE
\begin{displaymath}
\begin{split}
\int_{J_{\gamma}} |  \boldsymbol\pi_{\gamma} {\, \cdot\,} \overline{\p}_{\gamma}  | \, \di x & = \int_{J_{\gamma}} \frac{ d(z_{\gamma} ) }{ \sqrt{ | \p_{\gamma} |^{2} + \frac{1}{\gamma^{2}}}} \, \Bigg| |  \overline{\p}_{\gamma}  | ^{2} - \frac{ ( \p_{\gamma} {\, \cdot\,}  \overline{\p}_{\gamma}  )^{2}}{ | \p_{\gamma} |^{2} + \frac{1}{\gamma^{2}}} \Bigg|\, \di x 
\\
&
\geq  \frac{\gamma \lambda}{\sqrt{\gamma + 1}} \int_{J_{\gamma}} \Bigg| |  \overline{\p}_{\gamma}  | ^{2} - \frac{( \p_{\gamma}{\, \cdot \,}  \overline{\p}_{\gamma} )^{2}}{ | \p_{\gamma}|^{2} + \frac{1}{\gamma^{2}}}\Bigg| \, \di x \,.
\end{split}
\end{displaymath}
  Hence, from the boundedness  of~$ \boldsymbol\pi_{\gamma}$ and~$ \overline{\p}_{\gamma} $ in~$L^{2}(\Om; \M^{n}_{D})$ we deduce that\EEE
\begin{equation}\label{e:46}
\lim_{\gamma \to +\infty} \int_{J_{\gamma}}  \Bigg| |  \overline{\p}_{\gamma}  | ^{2} - \frac{( \p_{\gamma}{\, \cdot \,}  \overline{\p}_{\gamma}  )^{2}}{ | \p_{\gamma}|^{2} + \frac{1}{\gamma^{2}}}\Bigg| \, \di x = 0\,.
\end{equation}
Combining~\eqref{e:42}--\eqref{e:46} we get that
\begin{displaymath}
 \int_{\Om} |  \overline{\p}  | \big( d(z ) - | \boldsymbol\rho| \big) \, \di x = 0
 \end{displaymath}
 which implies $|  \overline{\p}  | = 0 $ a.e.~in~$\{ | \boldsymbol\rho| \leq  d(z ) \}$. This concludes the proof of the theorem.
\end{proof}

\section*{Acknowledgment}
US is partially supported by the FWF projects F\,65, I\,2375, and  P\,27052 and by the Vienna Science and Technology Fund (WWTF) through Project MA14-009.

\bibliographystyle{siam}

\end{document}